\documentclass[11pt,reqno,a4paper]{amsart}
\usepackage[normalem]{ulem}
\usepackage[colorlinks=true,linkcolor=blue,citecolor=magenta]{hyperref}
\usepackage{amssymb,epsfig,graphics}
\usepackage{amsmath}
\usepackage{enumerate}
\usepackage{tikz}
\usetikzlibrary{arrows}
\usepackage{bbm}       

\newcommand{\vertiii}[1]{{\left\vert\kern-0.25ex\left\vert\kern-0.25ex\left\vert #1 
    \right\vert\kern-0.25ex\right\vert\kern-0.25ex\right\vert}}

\newtheorem{theorem}{Theorem}[section]
\newtheorem{proposition}[theorem]{Proposition}
\newtheorem{definition}[theorem]{Definition}

\newtheorem{lemma}[theorem]{Lemma}

\newtheorem{remark}[theorem]{Remark}

\theoremstyle{definition}

\def\C{\mathcal{C}}

\def\RR{\mathbb{R}}
\def\TT{\mathbb{T}}

\DeclareMathOperator{\divv}{div}

\newcommand\vf{\varphi}

\def\RR{{\mathbb R}}

\def\1{{{\mathit 1} \!\!\>\!\! I} }

\DeclareMathOperator{\Leb}{Leb}

\numberwithin{equation}{section}
\newcommand\cA{{\mathcal A}}
\newcommand\cB{{\mathcal B}}
\newcommand\cC{{\mathcal C}}
\newcommand\cD{{\mathcal D}}

\newcommand\cH{{\mathcal H}}
\newcommand\cK{{\mathcal K}}

\newcommand\cL{{\mathcal L}}
\newcommand\cO{{\mathcal O}}
\newcommand\cP{{\mathcal P}}
\newcommand\cM{{\mathcal M}}

\newcommand\cR{{\mathcal R}}
\newcommand\cS{{\mathcal S}}
\newcommand\cT{{\mathcal T}}
\newcommand\cU{{\mathcal U}}

\newcommand\bB{{\mathbb B}}
\newcommand\bC{{\mathbb C}}

\newcommand\bN{{\mathbb N}}

\newcommand\bQ{{\mathbb Q}}
\newcommand\bR{{\mathbb R}}
\newcommand\bT{{\mathbb T}}

\newcommand\bV{{\mathbb V}}
\newcommand\bW{{\mathbb W}}
\newcommand\bZ{{\mathbb Z}}

\newcommand\bh{{\mathbbm h}}

\newcommand{\ve}{\varepsilon}
\newcommand\Id{{\mathbbm{1}}}

\newcommand{\Const}{C_{\#}}
\newcommand{\const}{c_{\#}}

\begin{document}
\title{Mean field coupled dynamical systems: Bifurcations and phase transitions}
\author{Wael Bahsoun}
\address{Department of Mathematical Sciences, Loughborough University,
Loughborough, Leicestershire, LE11 3TU, UK}
\email{W.Bahsoun@lboro.ac.uk}
\author{Carlangelo Liverani}
\address{
Dipartimento di Matematica\\
II Universit\`{a} di Roma (Tor Vergata)\\
Via della Ricerca Scientifica, 00133 Roma, Italy.}
\email{{\tt liverani@mat.uniroma2.it}}
\address{
Department of Mathematics\\
William E. Kirwan Hall, 4176 Campus Dr, University of Maryland\\
College Park, MD 20742, USA}
\email{{\tt cliveran@umd.edu}}
\thanks{ \tiny The research of W. Bahsoun (WB) was supported by EPSRC grant EP/V053493/1. The research of C. Liverani (CL) was supported by the PRIN Grants ``Regular and stochastic behavior in dynamical systems" (PRIN 2017S35EHN), ``Stochastic properties of dynamical systems" (PRIN 2022NTKXCX) and by the MIUR Excellence Department Project Math@TOV awarded to the Department of Mathematics, University of Rome Tor Vergata. CL acknowledges the membership of GNFM/INDAM. CL thanks the Department of Mathematical Sciences of the University of Loughborough for supporting his visit during the completion of this work. WB thanks the Department of Mathematics at the University of Rome Tor Vergata for its hospitality during the course of this work. WB and CL would like to thank the hospitality of MATRIX at the University of Melbourne and the School of Mathematics and Statistics at University of New South Wales where part of this research was conducted. Our participation in the MATRIX programme was supported by the Simons Foundation. C.L. thanks S\'ebastien Gouëzel for a helpful discussion. Finally, we thank Stefano Galatolo, for pointing out a mistake in Lemma 5.1 of a previous version, and the anonymous referee for helpful remarks.
}
\begin{abstract}
We develop a bifurcation theory for infinite dimensional systems satisfying abstract hypotheses that are tailored for applications to mean field coupled chaotic maps. Our abstract theory can be applied to many cases, from globally coupled expanding maps to globally coupled Axiom A diffeomorphisms. To illustrate the range of applicability, we analyze an explicit example consisting of globally coupled Anosov diffeomorphisms. For such an example, we classify all the invariant measures as the coupling strength varies; we show which invariant measures are \emph{physical}, and we prove that the existence of multiple invariant physical measures is a purely infinite dimensional phenomenon, i.e., the model exhibits phase transitions in the sense of statistical mechanics.
\end{abstract}
\date{\today}
\maketitle
\markboth{W. Bahsoun and C. Liverani}{Mean field coupled dynamical systems}
\bibliographystyle{plain}
\tableofcontents
\section{Introduction}
The study of large, possibly infinite-dimensional, coupled systems is of growing importance in many fields, including mathematics, physics and biology, see for example \cite{BA11, Bi21, P20, Y13} and references therein.
A far-reaching class of coupled systems is the \emph{globally coupled} one, also called \emph{mean field} systems, of interacting units in which each unit interacts with all other units of the system. In the field of dynamical systems, such systems provide an interesting alternative to locally coupled map lattices that were investigated by many authors, starting with Kaneko \cite{Ka83}, and were first put in a rigorous setting by Bunimovich and Sinaĭ \cite{BS88}. 

In physics, mean field systems are of paramount importance and have been studied for a long time. They are described by the Vlasov equation (see \cite{SP12}), which is far from being completely understood as it may exhibit unexpected phenomena, e.g. Landau damping \cite{MV}. Mean field systems also appear in the study of nonlinear Markov processes \cite{Ko10}.

Ershov, Potapov \cite{EP95} and Keller \cite{K00} first investigated mean filed systems when the dynamics of each unit is described by an expanding map, and introduced the idea that, in the infinite size limit, a \emph{nonlinear} transfer operator can describe the evolution of states of the underlying system. This plays a role similar to the Vlasov equation for gas models. 

Following the work of Keller \cite{K00}, a considerable amount of work has been done on globally coupled systems where  the dynamics of each unit is described by an expanding map \cite{BKST18, Blank, ST21, G21},\footnote{The work of \cite{G21} contains a general framework when the site dynamics admits exponential
decay of correlations and also applies to certain coupled random systems.} and more recently by an Anosov map \cite{BLS23} or by an intermittent map \cite{BK23} (see also \cite{T22} for a recent review). 
Most of the above results pertain to the case of \emph{weak} coupling strength where the existence of a unique invariant measure with appropriate characteristics (e.g. being \emph{physical}) is proven. However, there are few results in the literature in which it is proven that a globally coupled system may admit several invariant measures (\cite{BKZ,Se21}). Indeed, in \cite{BKZ} an example of a globally coupled expanding map where the infinite system admits exactly three absolutely continuous invariant measures is discussed. Two of the invariant measures of the infinite system are \emph{stable}; i.e., physical in the strong sense of this paper, while the third one is unstable. Moreover, it is shown that in the finite-dimensional case, the system preserves exactly one absolutely continuous invariant measure. However, the example is ad hoc and is explicitly constructed to exhibit multiple invariant measures. A general theory to explore all the invariant measures of given systems is missing.\\
Our goal is to start developing a general theory describing the invariant measures and their characteristics for globally coupled systems when the coupling strength is not necessarily weak, but the system still retains a chaotic behavior. Since such a problem can be reduced to the study of the fixed points of an infinite dimensional dynamical system described by a nonlinear operator acting on an appropriate Banach space, it is natural to wonder if some elementary facts of bifurcation theory can be extended to this context. 

In this work, we first develop a general theory to study bifurcations in infinite dimensional systems satisfying some abstract hypotheses. These hypotheses are not the standard ones under which the bifurcation theory of infinite dimensional systems is studied, they are tailored for applications to globally coupled chaotic maps. We discuss the simplest bifurcation: a saddle-node in a one-parameter family. 
However, we expect that the theory can be generalized to include other types of bifurcations and multi-parameter families of systems, as we believe this to be the beginning of a more general theory.

To show that the theory we develop applies to cases of interest, we study a simple and natural example in which we consider globally coupled Anosov diffeomorphisms. In this example, we show that in the weak coupling regime, the system admits a unique physical measure, while for larger coupling strength the system admits multiple physical measures (see Proposition \ref{prop:exsols} for a precise statement). More importantly, we show that the latter aspect is only an infinite-dimensional phenomenon; i.e., it appears only in the infinite-dimensional limit (thermodynamical limits). Hence, we are in the presence of phase transitions in the sense of statistical mechanics \cite{R69,Si14}, and the bifurcation diagram can be viewed as a phase transition diagram.

It is worth noting that our abstract theory can be applied to other classes of infinite dimensional systems, e.g. globally coupled expanding maps, general Axiom A maps, and partially hyperbolic maps for which one has an adequate understanding of the transfer operator (possibly \cite{CL22}), but we believe that our example suffices to illustrate the range of applicability of our abstract results.

The paper is organized as follows: In section \ref{sec:bifo}, we present an abstract bifurcation theory adapted to our needs, 
then we discuss a general class of globally coupled systems, and we state the main abstract results of our work (Theorem \ref{thm:implicit}, Theorem \ref{thm:implicitg} and Theorem \ref{lem:physical}). In Section \ref{sec:regularity} we study the differentiability properties of the nonlinear transfer operator. Section \ref{sec:biproof} contains the proof of the two first main theorems on the existence of the invariant measures and their dependence on the coupling strength. Section \ref{sec:physical} contains the proof of  Theorem \ref{lem:physical} and discusses the properties of the above invariant measures: in particular, it provides a criterion for identifying physically observable invariant measures.
Section \ref{sec:example} discusses in detail an example showing the presence of a large interval of coupling strength for which there exist many physical invariant measures. It is important to note that, contrary to most of the literature on the subject, we can use our general theory to treat coupling strength well beyond the perturbative regime. Moreover, we show that the existence of many invariant states is a purely infinite dimensional phenomenon (section \ref{sec:phase}); that is, we are in the presence of a phase transition.
Finally, in Appendix \ref{sec:implicit}, we provide an explicit quantitative version of the infinite-dimensional implicit function theorem since we had trouble locating it in the literature. Indeed, such a theorem is often stated in non-quantitative terms, while we need an estimate of the domain in which the theorem applies.
\newpage
\section{The model and an abstract bifurcation theory}\label{sec:bifo}
We are interested in the following class of dynamical systems.\\
Let $T\in\cC^\infty(\TT^d,\TT^d)$.\footnote{ Here, and in the following, we use $\C^\infty$ to simplify notation, in fact $\cC^r$, with e.g. $r\geq 6$, would suffice. Also, the theory can be easily extended to maps on a compact manifold $M$, yet we decided to keep things simple.}   
For $\bar r\in\bN$, $\bar r>12$, $h\in \cC^{\bar r}(\bT^d,\bR)'$,\footnote{ By $\cC^{\bar r}(\bT^d,\bR)'$ we mean the dual space of $\cC^{\bar r}(\bT^d,\bR)$, which is a Banach space when equipped with the norm $\|h\|=\sup_{\|\vf\|_{\cC^{\bar r}}\leq 1} |h(\vf)|$, e.g. see \cite[Defintion 24.3]{treves} where the same space is called ${\cD'}^{\bar r}(\bT^d)$.} $h(1)=1$, and $\nu\in\bR$ let
\begin{equation}\label{eq:coupledmap}
T_{\nu,h}(x)=T(x)+\nu\beta(x)h(\alpha) \quad\text{mod }1,
\end{equation}
where $\alpha\in C^\infty(\TT^d,\bR)$ and $\beta\in C^\infty(\TT^d,\RR^d)$. 
\begin{remark}
In the case when $h$ is a probability measure and $\beta=\bar\beta\circ T$, definition \eqref{eq:coupledmap} reads
\[
\begin{split}
&T_{\nu,h}(x)=\cS_{\nu,h}\circ T\\
&\cS_{\nu,h}=x+\nu\bar \beta(x)\int_{\bT^d}\alpha(y)h (dy).
\end{split}
\]
Or, if $h$ is absolutely continuous with a density that, by abusing notation, we still call $h$, $\cS_{\nu,h}=x+\nu\bar \beta(x)\int_{\bT^d}\alpha(y)h (y) dy$.
This is the same type of maps studied by \cite{EP95, K00}.
\end{remark}

In the case when $h$ is a probability measure we call the system in \eqref{eq:coupledmap} a mean field coupled system, or globally coupled map, of infinitely many interacting particles with coupling strength $\nu$ and state $h$. See \cite{EP95, K00, BLS23} for an explanation of this language, physical motivations, and references.

For simplicity, we restrict ourselves to systems of type \eqref{eq:coupledmap}, but we believe that our approach could be applied to the more general case
\[
T_{\nu,h}(x)=T(x)+\nu\int_{\bT^d}K(x,y)h(dy)\quad\text{mod }1
\]
It should also be possible to study the case in which the mean field influences a parameter of the map via a feedback function in the spirit of \cite{BKZ}.

\begin{remark} From now on, given a distribution $h\in \cC^{\bar r}(\bT^d,\bR)'$, we write, for each $\vf\in\cC^{\bar r}(\bT^d,\bR)$,  $h(\vf)=\int_{\bT^d}\vf h dm$, where $m$ is the Lebesgue measure. This debatable notation, borrowed from physics, has the advantage that if $h\in L^1(\bT^d)$, and one considers the continuous embedding of functions into the space of distributions given by $h\hookrightarrow h m$, then the notation expresses only the abuse of notation in which objects related by the embedding are identified. This is very convenient since many arguments are first carried out for functions and then extended to distributions by density (see the last item in \eqref{eq:embedding}).
\end{remark}
Let $\cL_{S}\in L(\cC^{\bar r}(\bT^d,\bR)',\cC^{\bar r}(\bT^d,\bR)')$ denote the transfer operator associated with a map $S\in\cC^\infty(\bT^d,\bT^d)$:\footnote{Given two Banach spaces $X,Y$, by $L(X,Y)$ we mean the Banach space of the bounded linear operators from $X$ to $Y$ equipped with the norm $\|A\|=\sup_{\|v\|_X\leq 1}\|Ay\|_Y$.} for $h\in \cC^{\bar r}(\bT^d,\bR)', g\in \cC^{\bar r}(\bT^d,\bR)$
\begin{equation}\label{eq:L_def}
\int_{\TT^d}g\cdot \cL_{S}hdm:=\int_{\TT^d}g\circ S\cdot hdm.
\end{equation}
In particular, $\int_{\TT^d}\cL_{S}h dm=\int_{\TT^d}h dm$. \\
\begin{remark} 
Note that, with our definition, the transfer operator acting on measures is nothing else than the push forward, so it can also be written as $(T_S)_* h$. We will occasionally use such a notation.
\end{remark}
If $S(\bT^d)=\bT^d$ and $\det(DS)\neq 0$, where $DS$ is the Jacobian matrix, then $\cL_S$ can be restricted to an element of $L( L^1(\bT^d), L^1(\bT^d))$, since for $h\in L^1(\TT^d)$ and $g\in \cC^{\bar r}(\TT^d)$
\begin{equation}\label{eq:L_def1}
\int_{\TT^d}g\cdot \cL_{S}hdm=\int_{\TT^d}g\circ S\cdot hdm=\int_{\bT^d} g(x)\sum_{S(y)=x} \frac{h(y)}{\det(D_yS)}dm(x).
\end{equation}
Let us define
\begin{equation}\label{eq:nonlinop}
\widetilde\cL_\nu(h)=\cL_{T_{\nu,h}}(h).
\end{equation}
\begin{remark} In the following we will consider the restriction of $\cL_S$, as defined in \eqref{eq:L_def}, and of $\widetilde\cL_\nu$, as defined in \eqref{eq:nonlinop}, to various invariant subspaces of the space of distributions. Since this does not create confusion, we will slightly abuse notations and call all the restrictions with the same name, i.e., $\cL_S$ and $\widetilde\cL_\nu$, respectively.
\end{remark}

\begin{remark}
Note that the restriction of the operator $\widetilde\cL_\nu$ on measures corresponds to the operator $T_\epsilon$ in \cite[Section 2.3]{K00} (we call $\nu$ what in \cite{K00} is called $\epsilon$).
\end{remark}

As an example, calling $\bW_0$ the space of probability measures, $(\bW_0,\widetilde\cL_\nu)$ is an infinite dimensional, nonlinear, dynamical system.
The fixed points of $\widetilde\cL_\nu$ are invariant states of the coupled system \eqref{eq:coupledmap}. Our goal is to study the number and properties of solutions of the equation
\begin{equation}\label{eq:fixed}
\widetilde\cL_\nu(h_\nu)=h_\nu,
\end{equation}
as $\nu$ varies.

\begin{remark} Note that $\widetilde\cL_\nu$ could have a lot of fixed points in which we are not interested, for example $\delta$ functions.
From now on, we will be interested only in solutions of \eqref{eq:fixed} for which, $h_\nu$ is an SRB measure for the map $T_{\nu,h_\nu}$. This is made clear in equation \eqref{eq:implicit}. We will not explicitly mention this limitation anymore.
\end{remark}

To study the solution of \eqref{eq:fixed}, it is convenient, and often necessary, to investigate the action of $\cL_{T_{\nu,h}}$ on appropriate Banach spaces since the invariant measure may not be absolutely continuous with respect to Lebesgue. 

Since our aim is to develop a general theory, we study \eqref{eq:fixed} under some abstract assumptions on the existence of Banach spaces with certain properties. Such hypotheses hold for a large class of maps, e.g., expanding maps and, most notably, Anosov maps. The latter case will be discussed explicitly in section \ref{sec:example} where a concrete example is analyzed in detail.\\
Let $\cT\subset\cC^{\bar r}(\bT^d,\bT^d)$ be a set of surjective maps closed by compositions. Assume that there exists $\nu_*>0$ such that,  for all $\nu\in[0,\nu_*]$ and  $ h\in \bW_0 $, $T_{\nu,h}\in\cT$.\\
Let  $\{\cB_i\}_{i=0}^3$ be Banach spaces such that, for some $\bar r\in\bN$,\footnote{ The inclusion is meant in the sense that there exists a one-to-one embedding.  As already stated, the embedding in $(C^{\bar r}(\TT^d))'$ is the natural extension of the embedding of $\iota:C^{\bar r}(\TT^d)\to ((C^{\bar r}(\TT^d))'$ given by $\iota(h)(\phi)=\int h\phi$, for each $\phi\in C^{\bar r}(\TT^d)$. Also, by $\cB_{i+1}\subset \cB_{i}, i=0,1, 2$, we mean $\|\cdot\|_{\cB_i}\le\|\cdot\|_{\cB_{i+1}}$.} 
\begin{equation}\label{eq:embedding}
\begin{split}
&C^{\bar r}(\TT^d,\bR)'\supset \cB_0\supset\cB_1\supset \cdots\supset  \cB_3 \supset C^{\bar r}(\TT^d,\bR)
\\
&C^{\bar r}(\TT^d,\bR)'\supset \cB_3'\supset\cB_2'\supset \cdots\supset  \cB_0'\supset C^{\bar r}(\TT^d,\bR)\\
& C^{\bar r}(\TT^d,\bR)  \textrm{ is dense in each }\cB_i, \cB_i'.
\end{split}
\end{equation}
For $K_\star>0$ large enough, let\footnote{ By $h\geq 0$ we mean that for all $\vf\in\cC^{\bar r}$, $\vf\geq 0$, $\int h \vf\geq 0$.}
\[
\begin{split}
&\bW=\left\{h\in  \cB_1: \|h\|_{\cB_1}\leq K_\star, h\ge 0, \int h=1\right\}\\
&\bV=\left\{f\in\cB_1:\, \int f=0\right\}.
\end{split}
\]
For all $S\in\cT$,  we can then restrict $\cL_S$ to each $\cB_i$. As already mentioned, to ease notation we will call such restrictions $\cL_S$ as well.  If $\cL_S\in L(\cB_1,\cB_1)$, then equation \eqref{eq:L_def}  implies $\cL_{S}\bV\subset \bV$.\\
We assume that there exist $K,C_*\geq 1$, $\bar c, \hat c>0$, and $\vartheta\in (0,1)$ such that,
\begin{enumerate}[({A}1)]
\item  \label{ass:2}  For each $S\in\cT$, $\cL_{S}\in L(\cB_i,\cB_i)$ for all $i\in\{0,\dots,3\}$.
\item \label{ass:0} For each $k\in\{1,\dots,d\}$, and $i\in\{0,\dots, 2\}$, $\partial_{x_k}\in L(\cB_{i+1},\cB_{i}\}$.
\item \label{ass:1} For any $a\in \cC^{\bar c i+ \hat c}$, $\|a h\|_{\cB_i}\leq C_* \|a\|_{\cC^{\bar c i+\hat c}}\|h\|_{\cB_i}$,  $\left|\int h\right|\leq C_* \|h\|_{\cB_0}$. 
\item\label{ass:3} The unit ball of $\cB_{i+1}$ is relatively compact in $\cB_i$ for $i\in\{0,1,2\}$.
\item  \label{ass:6} For all $n\in\bN$, $\{S_j\}_{j=1}^n\subset \cT$ and $\phi\in\cB_i$, $i\in\{1,\dots,3\}$,
\[
\begin{split}
&\|\cL_{S_n}\cdots\cL_{S_1}  \phi\|_{\cB_0}\leq C_*\|\phi\|_{\cB_0}\\
&\|\cL_{S_n}\cdots\cL_{S_1}  \phi\|_{\cB_i}\leq C_*\vartheta^{n}\|\phi\|_{\cB_i}+K\|\phi\|_{\cB_{i-1}}.\\
\end{split}
\]
\item  \label{ass:4}   The exists $\bar n\in\bN$ such that, for all $\{S_j\}_{j=1}^{n}\subset \cT$ , $n\geq \bar n$, $i\in\{1,\dots,3\}$,
\[
\|\cL_{S_{n}}\cdots\cL_{S_1}|_{\bV\cap\cB_i}\|_{\cB_i}< \vartheta^{n}.
\]
\end{enumerate}

 \begin{remark}\label{rem:gap}
It is worth noting that (A\ref{ass:3}) and (A\ref{ass:6}) imply that the operators $\cL_S$ are quasi-compact. Assumption (A\ref{ass:4}) implies that $\cL_S$ has a spectral gap on $\cB_i$, $i=1,2, 3$, see Lemma \ref{lem:spectra}. Note that the existence of a spectral gap is weaker than Assumption (A\ref{ass:4}) since the compositions can comprise very different operators.
 \end{remark}
 The first implication of the above assumptions is the following.
 \begin{lemma}\label{lem:spectra}
 Under assumptions (A\ref{ass:6}) and (A\ref{ass:4}),
 for each $\{S_j\}_{j=1}^{\bar n}\subset \cT$ we have $\cL_{S_{\bar n}}\cdots\cL_{S_1}\bW\subset \bW$. In addition, for each $S\in \cT$, $\sigma_{\cB_1}(\cL_S)\subset \{1\}\cup\{z\in\bC\;:\;|z|<\vartheta\}$ and $1$ is a simple eigenvalue. Finally, if $h\in\cB_0$ is a probability measure such that $(T_S)_*h=h$, then $h\in\bW \cap \cB_3$.
 \end{lemma}
 \begin{proof}
 Note that, for each $h\in\bW$ and $\{S_j\}_{j=1}^{\bar n}\subset \cT$, assumption (A\ref{ass:6}) implies, for each $n\in\bN$,
 \begin{equation}
 \begin{split}
& \|\cL_{S_{ n}}\cdots\cL_{S_1}1\|_{\cB_1}\leq C_*\|1\|_{\cB_1}+K\|1\|_{\cB_{0}}=:K_1\\
&\|\cL_{S_{ n}}\cdots\cL_{S_1}h\|_{\cB_1}\leq C_*\|h\|_{\cB_1}+K\|h\|_{\cB_{0}}\leq (C_*+K)K_\star.
\end{split}
 \end{equation}
 Then Assumption (A\ref{ass:4}) implies 
 \begin{equation}\label{eq:K-inv}
 \begin{split}
\|\cL_{S_{\bar n}}\cdots\cL_{S_1}h\|_{\cB_1}&\leq \|\cL_{S_{\bar n}}\cdots\cL_{S_1}(h-1)\|_{\cB_1}+K_1\\ 
&\leq \vartheta^{\bar n}\|h-1\|_{\cB_1}+K_1\leq \vartheta^{\bar n}\|h\|_{\cB_1}+2K_1\\
&\leq K_\star\vartheta+2K_1\leq K_\star
\end{split}
\end{equation}
provided $K_\star$ is large enough. Hence, $\cL_{S_{\bar n}}\cdots\cL_{S_1}\bW\subset \bW$.\\
To prove the second statement, note that $\|\cL_{S}^{\bar n}|_{\bV}\|_{\cB_1}< \vartheta^{\bar n}$, hence $\sigma(\cL_S|_{\bV})\subset \{z\in\bC\;:\;|z|<\vartheta\}$. The lemma follows since $\bV$ has codimension one.
To conclude, the lemma assumes that $h$ is a probability measure and $(T_S)_*h=h$. Then, for all $\ve>0$, by density there exist $h_\ve\in\cC^\infty$, $\int h_\ve=1$, such that, for all $\vf\in\cC^0$,
\[
\left|\int \vf dh-\int \vf h_\ve dx\right|\leq \ve\|\vf\|_\infty. 
\]
It follows that 
\[
\left|\int \vf\, d(T_S)_*^nh-\int \vf \cL_S^n h_\ve \right|\leq \ve\|\vf\|_\infty.
\]

Since $\cL_S\in L(\cB_3,\cB_3)$ has 1 as a simple eigenvalue, whose eigenfunction we denote by $h_S\in\cB_3$, and a spectral gap, for each $\vf\in\cC^\infty$, $\|\vf\|_\infty\leq 1$, we have
\[
\int \vf d h=\lim_{n\to\infty}\int \vf d(T_S)_*^nh=\lim_{n\to\infty}\int \vf \cL_S^nh_\ve+\cO(\ve)=\int \vf h_S+\cO(\ve).
\]

Since $\ve$ is arbitrary by \eqref{eq:embedding} it follows $h=h_S$.\footnote{Here, we are implicitly identifying measures and functions in a canonical way, as previously explained.} In addition, arguing as in \eqref{eq:K-inv}, implies 
\[
\|h_S\|_{\cB_1}=\lim_{k\to\infty}\|\cL^{k\bar n}_Sh_S\|_{\cB_1}\leq \lim_{k\to\infty}\vartheta^{k\bar n}\|h_S\|_{\cB_1}+2K_1\leq K_\star,
\]
so $h\in\bW$.
\end{proof}
Another simple consequence of the above assumptions is the existence of periodic orbits of the nonlinear transfer operator $\widetilde\cL_\nu$; i.e., $\widetilde\cL_\nu^{\bar n} h_0=h_0$ for some $h_0\in\cB_1$ and some $\bar n\in\mathbb N$. Note that this implies that the measure $h_0$ is periodic with respect to the coupled dynamics.

\begin{theorem}\label{thm:period}
Under our assumptions, for each $\nu\in[0,\nu_*]$, $\widetilde\cL_\nu$ has at least one periodic orbit in $\cB_1$.
\end{theorem}
\begin{proof}
By Lemma \ref{lem:spectra} we have $\widetilde \cL_\nu^{\bar n}(\bW)\subset \bW$.
 Next, note that if $h,h_1\in\bW$ are close enough, Assumptions (A\ref{ass:1}), (A\ref{ass:6}), and  Lemma \ref{lem:regularity} imply
 \[
 \begin{split}
 \|\widetilde\cL_{\nu}(h)-\widetilde \cL_{\nu}(h_1)\|_{\cB_0}
& \leq  \|\cL_{T_{\nu,h}}(h)- \cL_{T_{\nu,h_1}}(h)\|_{\cB_0}+\|\cL_{T_{\nu,h_1}}(h-h_1)\|_{\cB_0}\\
&\leq (C_*+\bar K\|h\|_{\cB_1})\|h-h_1\|_{\cB_0}\leq (C_*+\bar K K_\star)\|h-h_1\|_{\cB_0},
 \end{split}
 \]
 for some constant $\bar K>0$.
 That is $\widetilde \cL_\nu$ is continuous on $\bW$ in the $\cB_0$ topology.
In addition, by Assumption (A\ref{ass:3}), $\bW$ is compact in $\cB_0$, and clearly also convex. Then, by the Schauder fixed point Theorem, there exists at least an $h_*\in\bW$ such that $\widetilde \cL_\nu^{\bar n} h_*=h_*$. 

Let $h_{*,k}:=\widetilde \cL_\nu^{k} h_*$ and $T_k=T_{\nu,h_{*,k-1}}\circ \dots\circ T_{\nu,h_{*,0}}$. Then $h_{*}=\cL_{T_{\bar n}}h_*$. Since $T_k\in\cT$ by assumption, Assumption \eqref {ass:4} implies that there exists a unique $h_0\in\cB_1$ such that $\cL_{T_{\bar n}}h_0=h_0$. Hence, it must be $h_*=h_0\in\cB_1$. The wanted periodic orbit is then $\{\widetilde \cL_\nu^k h_0\}_{k=0}^{\bar n-1}$.
\end{proof}
\begin{remark}\label{rem:fixed}
Notice that if $\bar n=1$, then the argument in the proof of Theorem \ref{thm:period} implies the existence of a fixed point for $\widetilde\cL_\nu$ in $\cB_1$.
\end{remark}
Next, we want to explore not periodic orbits, but rather fixed points of $\widetilde \cL_{\nu}$. In this direction, there are already many results on the existence of a unique fixed point of $\widetilde \cL_{\nu}$ for small $\nu$; see \cite{G21} for general results and references. \\
On the contrary, here we are interested in the non perturbative regime, i.e. large $\nu$.

Lemma \ref{lem:spectra} implies that $\cL_{T_{\nu,h}}$ has a unique invariant probability measure that belongs to $\bW$ (the SRB or physical measure). Let us call it $H(\nu,h)$. We will prove further properties of $H(\nu,h)$ in Lemma \ref{lem:H-prop} below. 

Let $\Omega:=[0,\nu_*]\times \bW$ equipped with the topology of $\bR\times\cB_0$. Note that $\Omega$ is then a relatively compact set. Our basic idea is to study solutions of the following implicit equation 
$$F=0,$$ 
where $F:\Omega\to\bV$ is defined by
\begin{equation}\label{eq:implicit}
F(\nu,h)=h-H(\nu,h).
\end{equation}
Note that $F(\nu,h)=0$ implies 
\[
h=H(\nu,h)=\cL_{T_{\nu,h}}H(\nu,h)=\cL_{T_{\nu,h}}h=\widetilde \cL_\nu(h)
\]
that is $h$ is a fixed point of $\widetilde \cL_\nu$. Moreover, if $\widetilde \cL_\nu h=h$, assumption (A\ref{ass:4}) implies $h=H(\nu,h)$.
We will see in Lemma \ref{lem:der} that, under our assumptions, $F\in\cC^2(\Omega,\cB_1)$. For convenience, we define  
\begin{equation}\label{eq:Theta}
\Theta(\nu, h)=-(\Id-\cL_{T_{\nu,h}})^{-1} \divv{\cL_{T_{\nu,h}}} \beta H(\nu,h),
\end{equation}
which will appear in the expression of the Fr\'echet derivative, $D_hF$, of $F$.
We are now ready to state our main results. First, we have a local result.
\begin{theorem}[Local]\label{thm:implicit}
Suppose that there exists $\nu_0\in [0,\nu_*)$ and $h_0\in \bW$ such that
\begin{equation}\label{eq:solution}
 F(\nu_0, h_0)=0,
\end{equation} 
and either $D_{h_0} F$ is invertible or
\begin{align}
& \operatorname{Range}(D_{h_0} F)(\nu_0,\cdot)\oplus D_\nu F(\nu_0,h_0)=\bV; \tag{H1}\label{eq:range}\\
& \int\alpha D^2_hF|_{\nu_0,h_0}(\Theta(\nu_0,h_0),\Theta(\nu_0,h_0))\neq 0;\tag{H2}\label{eq:seconder}\\
&\int \alpha h_0\neq 0. \tag{H3}\label{eq:nontrivial}
 \end{align}
Then there exists $\delta>0$ and a differentiable (in the $\cB_1$ topology) map $(G,\bh)=(-\delta,\delta)\to [0,\nu_*]\times \bW$, $(G(0),\bh (0))=(\nu_0,h_0)$, such that all the solutions of the equation $F(\nu, h)=0$ in $\{(\nu,h)\in\bR\times \bW\;:\;  |\nu-\nu_0|+\|h-h_0\|_{ \cB_1}\leq \delta\}$ belong to the set  $\bh(-\delta,\delta)$. In addition, if $D_hF(\nu_0,h_0)$ is invertible, then for each $\nu$ in a neighborhood of $\nu_0$, $F(\nu,h)=0$ has a unique solution. While if $\nu_0>0$ and $D_hF(\nu_0,h_0)$ is not invertible, then there exists $\nu\in (0,\nu_*)$ for which the equation $F(\nu,h)=0$ has a least two solutions.
\end{theorem}
The proof of the above theorem is contained in Section \ref{sec:inv_meas}.
\begin{remark}
Note that Theorem \ref{thm:implicit} implies linear response; i.e., differentiability in the $\cB_1$ topology, with respect to the parameter $\nu$, of the invariant measures.
\end{remark}
Next, we establish a global result, which characterizes all the fixed points of $\widetilde\cL_\nu$ (the invariant measures $h_\nu$ solutions of \eqref{eq:fixed}).
\begin{theorem}[Global]\label{thm:implicitg}
Suppose that for each $\nu\in [0,\nu_*]$, $h\in \bW$ such that
\begin{equation}\label{eq:solutiong}
 F(\nu, h)=0,
\end{equation} 
either $D_{h} F$ is invertible or
\begin{align*}
& \operatorname{Range}(D_{h} F)(\nu,\cdot)\oplus D_\nu F(\nu,h)=\bV;\\
& \int\alpha D^2_hF|_{\nu,h}(\Theta(\nu,h),\Theta(\nu,h))\neq 0;\\
&\int \alpha h\neq 0.
 \end{align*}
Then there exist finitely many intervals $I_k\subset \bR$ and differentiable curves $(G_k,\bh_k):I_k\to [0,\nu^*]\times\cB_1$ such that all the $\cB_1$ fixed points of $\widetilde\cL_\nu$ belong to $\cup_k \bh_k(I_k)\subset \cB_3$. More precisely $\widetilde\cL_{G_k(y)}\bh_k(y)=\bh_k(y)$ for all $y\in I_k$. Moreover, $\sharp\{G_k^{-1}(\nu)\}<\infty$ for all $\nu\in [0,\nu_*]$.
Finally, if for some $h\in\bW$, $\nu\in (0,\nu_*)$, $F(\nu,h)=0$ and $D_{h} F$ is not invertible, then there exists $\nu_1\in (0,\nu_*)$ such that 
$\sharp\{\bigcup_{k}G_k^{-1}(\nu_1)\}\geq 3$ (that is, there are at least three invariant measures for $\widetilde \cL_{\nu_1}$).
\end{theorem}
The proof of the above Theorem can be found in section \ref{sec:invariant_global}.\\
According to Theorem \ref{thm:implicitg} the solutions $h_\nu$ of equations \eqref{eq:fixed}, for different $\nu$, can be all written as $h_{G_k(\tau)}=\bh_k(\tau)$ for certain maps $\{(G_k,\bh_k)\}$.
Figures \ref{NewBirfucation} and \ref{2BBirfucation} below provide two pictorial illustrations for possible bifurcation diagrams compatible with Theorem \ref{thm:implicitg} when only one (figure 1) or two (figure 2) such maps are present.
\begin{figure}[ht]
   \centering
   \includegraphics [scale=0.44]{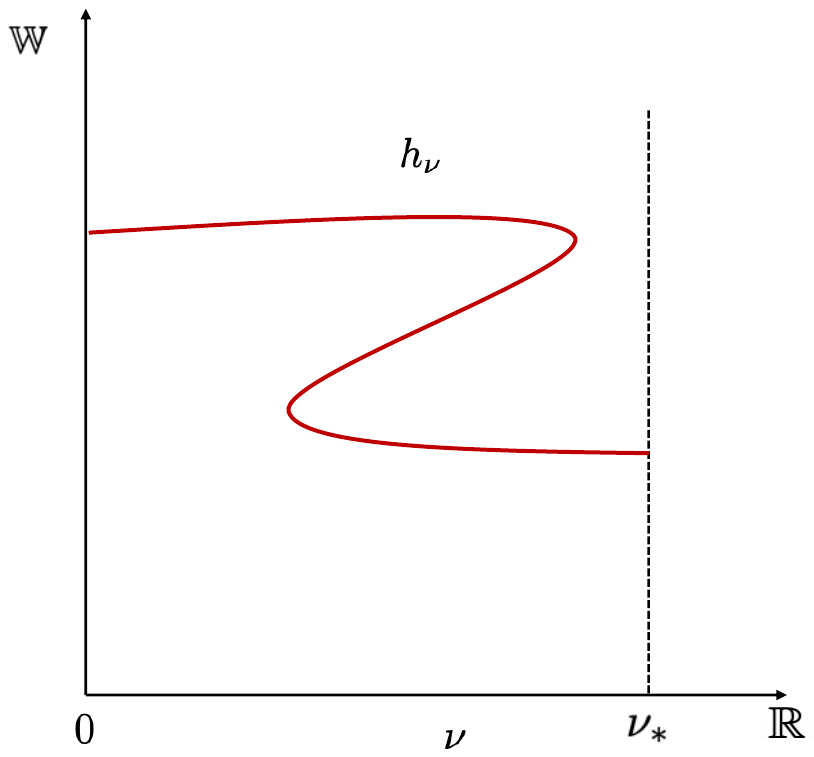} 
   \vskip -2cm
    \caption{A possible bifurcation diagram in Theorem \ref{thm:implicitg}.
     The red curve, $(G(\tau),\bh(\tau))$, is the set of the invariant measures $h_\nu$ (which satisfy the relation $h_{G(\tau)}=\bh(\tau)$).}   \label{NewBirfucation}
\end{figure} 
\begin{figure}[ht]
   \centering
   \includegraphics [scale=0.44]{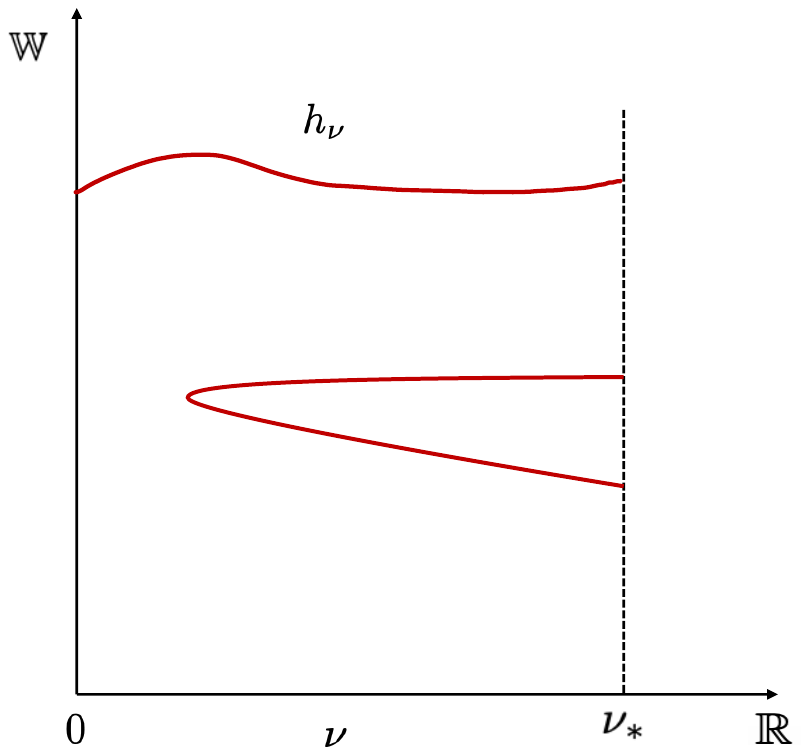} 
   \vskip -2cm
    \caption{Another possible scenario for the bifurcation diagram in Theorem \ref{thm:implicitg} in which two curve of solutions exist.}   \label{2BBirfucation}
\end{figure} 

We conclude this section with a discussion and a theorem concerning the physical relevance of the measures appearing in Theorems \ref{thm:implicit} and  \ref{thm:implicitg}.
In \cite{BLS23}, we defined a physical measure as one that can be realized starting from an absolutely continuous initial measure. However, such a definition does not capture the intuitive meaning of physical measure completely, which is that such measures describe the asymptotic behavior of a large class of initial conditions, and small changes in the initial conditions do not change the asymptotic measure. In the case at hand, the initial condition is a measure. Hence, it is natural to consider an absolutely continuous initial measure, in analogy with the finite-dimensional case, and this was the motivation of the definition in \cite{BLS23}.  Yet, it is also natural to ask that a small change in the initial measure does not change the asymptotic behavior of the system.\\
Hence, we will call the measures that satisfy the condition in  \cite{BLS23} {\em visible}, and we introduce the following new definition to characterize physical measures that include such a natural stability property. This is inspired, but different, from the idea of stability used in \cite{BKZ, Se21}. Let $d_w$ be a metric in the space of measures that metrizes the weak topology.

\begin{definition}\label{def:visible}
An invariant measure $h$ of $T_{\nu,h}$, that is $(T_{\nu,h})_*(h)=h$, is called a {\em visible} measure if there exists $g\in\cC^\infty$ such that 
\[
\lim_{n\to\infty}d_w(\widetilde \cL_\nu^n(g),h)=0.
\]
\end{definition}
\begin{definition}\label{def:physical}
An invariant measure $h$ of $T_{\nu,h}$, that is $(T_{\nu,h})_*(h)=h$, is called a {\em physical} measure if there exists an open set  $\cA\subset\cC^\infty(\bT^d,\bR_+)\setminus\{0\}$ such that, for any sequences $m_n(\mu)=(T_{\nu,m_{n-1}(\mu)})_{*}m_{n-1}(\mu)$, $n\in\bN$, $m_0(\mu):=\mu\in \cM_{\cA}$, where
\[
\cM_{\cA}=\left\{\mu:\; \frac{d\mu}{d\operatorname{Leb}}=\rho,\; \rho=\frac{g}{\int g}\;:\; g\in \cA\right\},
\] 
we have
\[
\lim_{n\to\infty} \sup_{\mu\in \cM_{\cA}}d_w(m_n(\mu), h)=0.
\]
\end{definition}

\begin{remark}
It is worth noting that the unique physical measures identified in \cite{BLS23} are indeed physical also in the stronger sense of Definition \ref{def:physical}, see  in \cite[Lemma 3.13]{BLS23}. 
\end{remark}
For  $z\not\in\sigma_{\cB_1}(\cL_{T_{\nu,h}})\setminus \{1\}$,\footnote{ By $\sigma_\cB(\cL)$ we mean the spectrum of $\cL\in L(\cB,\cB)$.} define the functions
\begin{equation}\label{eq:Thetaz}
\begin{split}
\Theta(z)&:=-\nu (z\Id-\cL_{T_{\nu,h}})^{-1}\divv \cL_{ T_{\nu,h}}\beta h\\
\Xi(z)&=\int\alpha \Theta(z).
\end{split}
\end{equation} 
\begin{theorem}\label{lem:physical}
For $\nu\in[0,\nu_*]$, if $h$ is a visible invariant measure, then $h\in\cB_3$. Moreover, $\widetilde \cL_\nu$ has a G\^ateaux differential $\cD_h$ at $h$, $\cD_h\in L(\cB_1,\cB_1)$ and it reads
\[
\cD_h(\phi)=\cL_{T_{\nu,h}}\phi-\nu\divv \cL_{ T_{\nu,h}}\beta h\int\alpha\phi.
\]
Finally, if $\sigma(\cD_h)\setminus \{1\}\subset \{z\in\bC\;:\; |z|<\kappa\}$ for some $\kappa\in (0,1)$, and $\Xi(1)\neq 1$, then $h$ is a physical measure.
On the contrary, if there exists $z_1\in\sigma (\cD_h)$, $|z_1|>1$, then $h$ is not a physical measure.
\end{theorem}
The proof of the above theorem can be found at the end of section \ref{sec:physical}.
\section{Regularity properties}\label{sec:regularity}
Our first task is to explore the regularity of maps $\cL_{S(u)}$, where $S$ is a map in $\cC^2(\bR,\cC^{\bar r}(\bT^2,\bT^2))$.\\
For each compact interval $U\subset\bR$ and $S\in\cC^2(U,\cC^{\bar r}(\bT^d,\bT^d))$, we set, for all $\psi\in\cC^\infty (\bT^2,\bT^2)$ and  $u\in U$
\begin{equation}\label{eq:derivP}
\cP_{u}\psi=-\divv \cL_{ S(u)}(S'(u)\psi).
\end{equation}

\begin{lemma}\label{lem:regularity}
There exists $C_\star>0$ such that, provided $\bar r$ is large enough, for each compact interval $U\subset\bR$ and $S\in\cC^2(U,\cC^{\bar r}(\bT^d,\bT^d))$, we have, for all $s\in \bR$, 
\[
\begin{split}
&\|\cP_{u}\psi\|_{\cB_{i}}\leq  C_\star\|S\|_{\cC^{\bar r}}\|\psi\|_{\cB_{i+1}}, \quad\forall \psi\in\cB_{i+1},\; i\in\{0,1,2\} \\
&\|(\cL_{S(u+s)}-\cL_{S(u)})\psi\|_{\cB_{i}}\leq  C_\star s\|S\|_{\cC^{\bar r}}\|\psi\|_{\cB_{i+1}},\quad\forall \psi\in\cB_{i+1}, \; i\in\{0,1,2\}\\
&\|(\cL_{S(u+s)}-\cL_{S(u)}-s\cP_{u})\psi\|_{\cB_{i}}\leq C_\star s^2\|S\|_{\cC^{\bar r}}^2\|\psi\|_{\cB_{i+2}},\quad\forall \psi\in\cB_{i+2},\; i\in\{0,1\}.
\end{split}
\]
\end{lemma}
\begin{proof}
The first inequality follows immediately from Assumption (A\ref{ass:0}), (A\ref{ass:1}) and  (A\ref{ass:6}).
Next, we study the continuity of $\cL_{S(\cdot)}:\bW\subset \cB_0\to\ L(\cB_i,\cB_i)$. For 
$\vf,\psi\in\cC^\infty$, we can compute
\begin{equation}\label{eq:der0}
\begin{split}
&\int\vf\left(\cL_{S(u+s)}-\cL_{S(u)}\right)\psi=\int\vf\circ S(u+s)\psi-\int\vf\circ S(u)\psi\\
&=\int_u^{u+s}\int\langle(\nabla\vf)\circ S(\tau),S'(\tau)\rangle\psi d\tau\\
&=\int_u^{u+s}\int\langle(\nabla\vf), \cL_{ S(\tau)}S'(\tau)\psi\rangle\\
&=\int_u^{u+s}\int\vf\{-\divv \cL_{ S(\tau)}S'(\tau)\psi\} ds.
\end{split}
\end{equation}
It follows that, for each $i\in\{0,1,2\}$, for some constant $C_\star$ large enoug, we have, recalling Assumptions (A\ref{ass:0}),  (A\ref{ass:6}) and (A\ref{ass:1}),
\[
\begin{split}
&\left|\int\vf\left(\cL_{S(u+s)}-\cL_{S(u)}\right)\psi\right|
\leq s\|\vf\|_{\cB_i'}\sup_{\tau\in[u,u+s]}\left\|\divv \cL_{ S(\tau)}S'(\tau)\psi\right\|_{\cB_i}\\
&\leq s\, C_\star C_*^{-2}\|\vf\|_{\cB_i'}\sup_{\tau\in[u,u+s]}\left\|\cL_{ S(\tau)}S'(\tau)\psi\right\|_{\cB_{i+1}}\\
&\leq s\, C_\star\|\vf\|_{\cB_i'}\|\psi\|_{\cB_{i+1}}\|S'\|_{\cC^{\bar c (i+1)+\hat c}}
\end{split}
\]
where $\cB_i'$ is the dual space of $\cB_i$.
Taking the sup on $\vf\in\cB_1'\cap \cC^{\infty}$, which is dense in $\cB_i'$ by \eqref{eq:embedding}, with  $\|\vf\|_{\cB_i'}\leq 1$, yields
\[
\left\|\cL_{S(u+s)}\psi-\cL_{S(u)}\psi\right\|_{\cB_i}\leq s\, C_\star\|\psi\|_{\cB_{i+1}}\|S\|_{\cC^{\bar r}}.
\]
We can continue \eqref{eq:der0} and write
\[
\begin{split}
&\int\vf\left(\cL_{S(u+s)}\psi-\cL_{S(u)}\psi+s\divv \cL_{ S(u)}S'(u)\psi\right)=\\
&=\int_u^{u+s}\int \langle(\nabla\vf)\circ S(\tau), S'(\tau)\rangle\psi- \langle(\nabla\vf)\circ S(u), S'(u)\rangle\psi\} d\tau\\
&=\int_u^{u+s}\hskip -12pt d\tau\int_u^{\tau}\hskip -6pt d\tau'\int  \langle(\nabla\vf)\circ S(\tau), S''(\tau')\rangle\psi
+D^2_{S(\tau')}\vf(S'(\tau'),S'(\tau))\psi\\
&=-\sum_{i=1}^d\int_u^{u+s}\hskip -12pt d\tau \int_u^{\tau}\hskip -6pt d\tau' \int \vf \partial_{x_i}\cL_{S(\tau)} (S_i''(\tau')\psi)\\
&\phantom{=}
+\sum_{i,j=1}^d\int_u^{u+s}\hskip -12pt d\tau \int_u^{\tau}\hskip -6pt d\tau' \int \vf\partial_{x_i}\partial_{x_j}\cL_{S(\tau')}(S_i'(\tau')S'_j(\tau)\psi),
\end{split}
\]
which implies, recalling Assumptions (A\ref{ass:0}), (A\ref{ass:1}) and (A\ref{ass:6}),
\[
\begin{split}
\left|\int\vf\left(\cL_{S(u+s)}\psi-\cL_{S(u)}\psi-s\cP_{u}\psi\right)\right|\leq &C_\star s^2
\|\vf\|_{\cB_i'}\Big[\|S''\|_{\cC^{\bar\alpha(i+1)+\bar \beta}}\|\psi\|_{\cB_{i+1}}\\
&+\|S'\|_{\cC^{\bar c(i+2)+\hat c}}^2\|\psi\|_{\cB_{i+2}}\Big].
\end{split}
\]
Taking the sup on $\vf\in\cB_i'\cap \cC^\infty$, with $\|\vf\|_{\cB_i'}$, yields, for $\bar r$ large enough,
\[
\left\|\cL_{S(u+s)}\psi-\cL_{S(u)}\psi-s\cP_{u}\psi\right\|_{\cB_i}\leq C_\star s^2
\|S\|_{\cC^{\bar r}}^2\|\psi\|_{\cB_{i+2}}.
\]
\end{proof}
\begin{remark}
Note that the third equation in the statement of Lemma \ref{lem:regularity} implies that
\begin{equation}\label{eq:derivLP}
\frac{d}{du}\cL_{S(u)}=\cP_u
\end{equation}
if we interpret $\cL_{S(u)}$ as a function from $\bR$ to $L(\cB_{i+2},\cB_i)$, $i\in\{0,1\}$.
\end{remark}
We can now use the above Lemma to investigate the regularity of the operator $\cL_{T_{\nu,h}}$ seen as a map from $\Omega$ to  $L(\cB_{i+2},\cB_i)$.
\begin{lemma}\label{lem:L-der}
$\cL_{T_{\cdot,\cdot}}\!\in\cC^1(\Omega, L(\cB_{i+2},\cB_i))$, $i\in\{0,1\}$, so it is Fréchect Differentiable. The derivative is given by, for each $h_1\in\bV$,
\begin{equation}\label{eq:basic_der}
\left[\left(\frac{\partial}{\partial h}\cL_{T_{\nu,h}}\right)\!(h_1)\right]\psi=\nu\{-\divv \cL_{ T_{\nu,h}}\beta\psi\}\int\alpha h_1.
\end{equation}

\begin{equation}\label{eq:basic_nuder}
\left(\frac{\partial}{\partial \nu}\cL_{T_{\nu,h}}\right)\psi=- \left(\divv\cL_{T_{\nu, h}}\beta \psi\right)\int \alpha h.
\end{equation}
\end{lemma}
\begin{proof}
For each $h_0\in\bW$, and $h_1\in\bV$ we can write $S(u)=T_{\nu, h_0+uh_1}$, $\|h_1\|_{\cB_0}=1$. Then, by \eqref{eq:coupledmap},
\[
S'(u)=\nu\beta\int\alpha h_1.
\]
Then we apply Lemma \ref{lem:regularity} to write
\[
\frac{d}{du}\cL_{S(u)}\psi=-\nu \left(\divv\cL_{T_{\nu, h_0+uh_1}}\beta \psi\right)\int \alpha h_1.
\]
Note that the above is the G\^ateaux derivative of $\cL_{T_{\nu,h}}$ with respect to $h$ in the direction $h_1$.

On the other hand, if we set $S(u)=T_{u,h_0}$, we have
\[
S'(u)=\beta \int \alpha h_0.
\]
Hence, Lemma \ref{lem:regularity} implies
\[
\frac{d}{du}\cL_{S(u)}\psi=- \left(\divv\cL_{T_{\nu, h_0+uh_1}}\beta \psi\right)\int \alpha h_0.
\]
Then the G\^ateaux derivatives $\frac{\partial}{\partial h}\cL_{T_{\nu,h}}$ and $\frac{\partial}{\partial_\nu}\cL_{T_{\nu,h}}$, by assumptions (A\ref{ass:0})  and the second of Lemma \ref{lem:regularity}, are continuous functions from $\Omega$ to $L(\cB_{i+2},\cB_i)$, $i\in\{0,1\}$. Hence, by \cite[Proposition 3.2.15]{DM13}, $\cL_{T_{\cdot,\cdot}}:\Omega\to L(\cB_{i+2},\cB_i)$ is Fréchect Differentiable with derivative given by \eqref{eq:basic_der} and \eqref{eq:basic_nuder}. 
\end{proof}

\subsection{Properties of the functions $H$ and $F$}\label{sec:Fregularity}\ \\
Our next goal is the study of the regularity of $F$, defined in \eqref{eq:implicit}, hence we have to start studying 
\[
H:\Omega\to\cB_i,
\]
where $\Omega:=[0,\nu_*]\times \bW$ is equipped with the topology of $\bR\times\cB_0$. Note that, by Assumption (A\ref{ass:3}), $\Omega$ is then a compact set.
Assumption (A\ref{ass:4}) implies
\begin{equation}\label{eq:specdeco}
\cL_{T_{\nu,h}}=\Pi_{T_{\nu,h}} +Q_{T_{\nu,h}},
\end{equation}
 as an operator on $\cB_i$, $i=1,2, 3$, where $\Pi_{T_\nu, h}$ is the one-dimensional projection corresponding to the eigenvalue 1 and $\Pi_{T_{\nu,h}}Q_{T_{\nu,h}}= Q_{T_{\nu,h}}\Pi_{T_{\nu,h}}=0$.
\begin{lemma}\label{lem:H-prop}
For each $(\nu,h)\in\Omega$, there exists a unique $H\in  \cB_1$ such that
\[
\Pi_{T_\nu,h}\phi=H(\nu,h)\int_{\bT^d} \phi,\quad\text{  and }\quad \|Q^n_{T_{\nu,h}}\|_{\cB_1}\le C\vartheta^n,
\]
for some $C>0$. In particular, 
\[
\int H(\nu,h)=1;\quad \cL_{T_{\nu,h}} H(\nu,h)=H(\nu,h), \textrm{ and   }Q_{T_{\nu,h}}H(\nu,h)=0.
\]  
Moreover, $H\in\cC_*^1(\Omega,\cB_1)\cap \cC^0(\Omega,\cB_2)\cap  L^\infty(\Omega,\cB_3)$,\footnote{ Here, $\cC_*^r$ means that the function is only G\^ateaux differentiable $r$ times, no assumption on the continuity of the derivatives is implied. We will use, as usual, $\cC^r$ for Fréchet differentiable functions with continuous differential. See \cite[section 3.2]{DM13} for the relation between these two differentials and their properties.} and there exists $\Const>0$, such that, for each $\ve>0$, $(\nu_1,h_1), (\nu_2,h_2)\in\Omega$, for which $|\nu_1-\nu_2|+\|h_1-h_2\|_{\cB_0}\leq \ve$, we have
\[
\|H(\nu_1,h_1)-H(\nu_2,h_2)\|_{\cB_2}\leq \Const \ve\ln\ve^{-1}.
\]
\end{lemma}
\begin{proof}
By \eqref{eq:L_def}, the Lebesgue measure is a left eigenfunction of $\cL_{T_{\nu,h}}$ with eigenvalue one. Hence, $\Pi_{T_\nu,h}\phi=H(\nu,h)\int_{\bT^d} \phi$ with $\cL_{T_{\nu,h}} H(\nu,h)=H(\nu,h)$ and $\int H(\nu,h)=1$. Given a small closed curve $\gamma$ around $1$ functional calculus yields \cite{Ka80}
\[
\Pi_{T_\nu, h}=\frac{1}{2\pi i}\int_\gamma (z-\cL_{T_{\nu,h}})^{-1} dz.
\]
By Lemma \ref{lem:regularity} and \cite[Section 8]{GL06}, or see \cite[Theorem 3.3]{Goez10} for a more precise version, we have
\[
\Pi_{T_\nu, h}\in \cC_*^1(\Omega,L(\cB_3,\cB_1))\cap \cC_*^0(\Omega,L(\cB_3,\cB_2)) .
\]
Accordingly, 
\begin{equation}\label{eq:H-der}
H\in \cC_*^1(\Omega,\cB_1)\cap \cC_*^0(\Omega,\cB_2).
\end{equation}
Arguing as in the proof of Lemma \ref{lem:spectra}, (A\ref{ass:6}) implies that $\|H(\nu,h)\|_{\cB_3}\leq\bar C$ for some $\bar C>0$.  The bound on $Q_{T_{\nu, h}}$ follows by Assumption (A\ref{ass:4}).

To prove the last fact, note that for each $n\geq \bar n$,  using Assumptions (A-\ref{ass:4}), (A-\ref{ass:6}), and Lemma \ref{lem:regularity},

\[
\begin{split}
&\|H(\nu_1,h_1)-H(\nu_2,h_2)\|_{\cB_2}= \|\cL_{T_{\nu_1,h_1}}^nH(\nu_1,h_1)-\cL_{T_{\nu_2,h_2}}^nH(\nu_2,h_2)\|_{\cB_2}\\
&\le \|(\cL_{T_{\nu_1,h_1}}^n-\cL_{T_{\nu_2,h_2}}^n)H(\nu_1,h_1)\|_{\cB_2}+\|\cL_{T_{\nu_2,h_2}}^n(H(\nu_1,h_1)-H(\nu_2,h_2))\|_{\cB_2}\\
&\le\sum_{k=1}^{n-1}\|\cL_{T_{\nu_1,h_1}}^k(\cL_{T_{\nu_1,h_1}}-\cL_{T_{\nu_2,h_2}})\cL_{T_{\nu_2,h_2}}^{n-k}H(\nu_1,h_1)\|_{\cB_2}+2\bar C \vartheta^n\\
&\leq  C_* C_\star\sum_{k=1}^{n-1}\ve\|\cL_{T_{\nu_2,h_2}}^{n-k}H(\nu_1,h_1)\|_{\cB_3}+2\bar C \vartheta^n\leq \Const \ve\ln\ve^{-1}
\end{split}
\]
provided we choose $n=\max\{\bar n,\frac{2}{\ln\vartheta^{-1}}\ln\ve^{-1}\}$ and $\Const$ is large enough. 
\end{proof}
Next, we provide an estimate needed in the following. Assumptions (A\ref{ass:6}) and (A\ref{ass:4}) allow to write
\[
\begin{split}
\|(\Id-\cL_{T_{\nu,h}})^{-1}|_{\bV}\|_{\cB_1}&\leq \sum_{k=0}^{\infty}\|\cL_{T_{\nu,h}}^k|_{\bV}\|_{\cB_1}\\
&\leq\sum_{k=0}^{\bar n-1} (C_*+K)+\sum_{k=\bar n}^\infty\vartheta^n\leq \bar n(C_*+K)+(1-\vartheta)^{-1}.
\end{split}
\]
Next, note that Lemma \ref{lem:regularity} and \cite[Section 8]{GL06} implies that there exists $C_0\geq  \bar n(C_*+K)+(1-\vartheta)^{-1}$ such that, for $i\in\{0,1,2\}$.
\[
\|[(\Id-\cL_{T_{\nu,h}})^{-1}-(\Id-\cL_{T_{\nu_0,h_0}})^{-1}]\psi\|_{\cB_i}\leq C_0 (|\nu-\nu_0|+\|h-h_0\|_{\cB_0})\|\psi\|_{\cB_{i+1}}.
\]

\begin{remark}\label{rem:theta}
The above estimate and the the definition \eqref{eq:Theta}, Lemmata \ref{lem:L-der}, \ref{lem:H-prop} and Assumption (A\ref{ass:0}) imply that $\Theta\in \cC^0(\Omega,\cB_{1})\bigcup\cC_*^1(\Omega,\cB_0)$. 
\end{remark}
\begin{lemma}\label{lem:der}
$F\in\cC^0(\Omega,\cB_2)\cap \cC^1(\Omega,\cB_1)\cap\cC^2(\Omega,\cB_0)$. Also,
for each $\phi\in \bV$, the following holds:
\[
\begin{split}
&{D_{h} F}|_{\nu, h}(\phi)=\phi-\nu\Theta(\nu,h)\int\alpha \phi
\end{split}
\]
while, if $h\in\cB_2$ and $\phi,\bar\phi\in\bV$, then
\[
\begin{split}
&{D_{h}^2 F}(\phi,\bar \phi)=-\Big\{2\sum_{i,j}\left(\Id-\cL_{T_{\nu,h}}\right)^{-1} \partial_{x_i} \left[\cL_{ T_{\nu,h}}\beta_i(\Id-\cL_{T_{\nu,h}})^{-1}\partial_{x_j}\left(\cL_{T_{\nu,h}}\beta_j  H(\nu,h)\right)\right]\\
&+(\Id-\cL_{T_{\nu,h}})^{-1}\sum_{i,j} \partial_{x_i}\partial_{x_j}\left[\cL_{T_{\nu,h}} \beta_i \beta_j H(\nu,h)\right] \Big\}\nu^2\int\alpha \phi\int\alpha \bar \phi.
\end{split}
\]
Finally, if $h\in\cB_1$,
\[
{\partial_\nu F}|_{\nu, h}=-\Theta(\nu,h)\int\alpha h.
\] 
\end{lemma}
\begin{proof}
By Lemma \ref{lem:L-der} $\cL_{T_{\cdot,\cdot}}$ is Fréchect Differentiable from $\cB_{i+2}$ to $\cB_i$, $i\in\{0,1\}$. Moreover, by Lemma \ref{lem:spectra}, $H(\nu, h)\in\cB_3$ and by \eqref{eq:H-der} it is G\^ateaux differentiable from $\Omega$ to $\cB_1$. 
Hence, for each $h\in\bW$, and $\phi\in\bV$ we have\footnote{ Here the limit is meant in the $\cB_1$ topology.}
\[
\begin{split}
\lim_{\ve\to 0}& \frac{\cL_{T_{\nu,h+\ve \phi}}H(\nu,h+\ve\phi)-\cL_{T_{\nu,h}}H(\nu,h)}\ve
=\lim_{\ve\to 0} \frac{\left[\cL_{T_{\nu,h+\ve \phi}}-\cL_{T_{\nu,h}}\right]H(\nu,h)}\ve\\
&+\lim_{\ve\to 0} \frac{\left[\cL_{T_{\nu,h+\ve \phi}}-\cL_{T_{\nu,h}}\right]\left[H(\nu,h+\ve\phi)-H(\nu,h)\right]}\ve\\
&+\lim_{\ve\to 0} \frac{\cL_{T_{\nu,h}}\left[H(\nu,h+\ve\phi)-H(\nu,h)\right]}\ve\\
&=[(\partial_h\cL_{T_{\nu,h}})H(\nu,h)](\phi)+\cL_{T_{\nu,h}}[\partial_hH(\nu,h)(\phi)],
\end{split}
\]
since the term in the limit of the second line is bounded by $\ve \ln\ve^{-1}$ due to the second assertion in the statement of Lemma \ref{lem:regularity}
and the last assertion in the statement of Lemma \ref{lem:H-prop}.

This gives the G\^ateaux differential of $\cL_{T_{\nu,h}}H(\nu,h)$ with respect to $h$.
Hence,
\begin{equation}\label{eq:Hderivh}
\begin{split}
\partial_h H(\nu,h)(\phi)&=(\Id-\cL_{T_{\nu,h}})^{-1}[(\partial_h\cL_{T_{\nu,h}})H(\nu,h)](\phi)\\
&=\nu\Theta(\nu,h)\int\alpha \phi.
\end{split}
\end{equation}
This yields the formula for the G\^ateaux derivatives $\partial_h H, \partial_h F\in L(\cB_0,\cB_1)$, which, since they are continuous for $h\in\bW$ with respect to the $\cB_0$ topology (see Remark \ref{rem:theta}), are the Fréchet derivative ($D_hH$ and $D_h F$, respectively) as well. Using a similar argument to the above, we also obtain
\begin{equation}\label{eq:Hderivnu}
\partial_\nu H(\nu,h)=\Theta(\nu,h)\int\alpha h.
\end{equation}
By \eqref{eq:Hderivh} and \eqref{eq:Hderivnu} and Remark \ref{rem:theta} it follows that $DH$, as an element of $L(\cB_0,\cB_0)$, is G\^ateaux differentiable, thus (recalling Lemma \ref{lem:H-prop}) we obtain $H\in\cC^0(\Omega,\cB_2)\cap \cC^1(\Omega,\cB_1)\cap\cC^2_*(\Omega,\cB_0)$.
To conclude, we must compute $D_h^2F$. Using \eqref{eq:basic_der}, and \eqref{eq:Hderivnu}, for every $\phi,\bar\phi\in \cC^\infty$ we have
\[
\begin{split}
&{D_{h}^2 F}(\phi,\bar \phi)=-\nu \lim_{\ve\to 0}\frac{\Theta(\nu,h+\ve\bar\phi)-\Theta(\nu,h)}\ve \int\alpha \phi \\
&=\nu \Big((\Id-\cL_{T_{\nu,h}})^{-1}D_h\cL_{T_{\nu,h}}(\bar\phi)\left[ (\Id-\cL_{T_{\nu,h}})^{-1}\divv \left(\cL_{T_{\nu,h}} \beta H(\nu,h)\right)\right]\\
&\phantom{=}
+(\Id-\cL_{T_{\nu,h}})^{-1}\divv \left[D_h\cL_{T_{\nu,h}}(\bar\phi) \beta H(\nu,h)\right]\\
&\phantom{=}
+(\Id-\cL_{T_{\nu,h}})^{-1}\divv \left[\cL_{T_{\nu,h}} \beta D_h H(\nu,h)(\bar\phi)\right] \Big)\int\alpha \phi.
\end{split}
\]
Then, using \eqref{eq:Theta}, \eqref{eq:basic_der} and \eqref{eq:Hderivh}, we have
\[
\begin{split}
&{D_{h}^2 F}(\phi,\bar \phi)=-\Big\{\sum_{i,j}\left(\Id-\cL_{T_{\nu,h}}\right)^{-1} \partial_{x_i} \cL_{ T_{\nu,h}}\beta_i(\Id-\cL_{T_{\nu,h}})^{-1}\partial_{x_j}\cL_{T_{\nu,h}}\beta_j  H(\nu,h)\\
&+(\Id-\cL_{T_{\nu,h}})^{-1}\sum_{i,j} \partial_{x_i}\partial_{x_j}\left[\cL_{T_{\nu,h}} \beta_i \beta_j H(\nu,h)\right]\\
&+\sum_{ij}(\Id-\cL_{T_{\nu,h}})^{-1}\partial_{x_i} \cL_{T_{\nu,h}} \beta_i (\Id-\cL_{T_{\nu,h}})^{-1}\partial_{x_j} \cL_{T_{\nu,h}} \beta_j H(\nu,h)\Big] \Big\}\nu^2\int\alpha \phi\int\alpha \bar \phi.
\end{split}
\]
To conclude, note that, by the above formula and Assumption (A\ref{ass:0}), the continuity of $D_{h}^2 F$ in the $\cB_0$ norm follows from the continuity of $H$ in the $\cB_2$ norm, which is stated in Lemma \ref{lem:H-prop}. Hence, the claimed Fréchet differentiability.
\end{proof}

\section{Characterization of invariant measures}\label{sec:biproof}
The basic idea is to apply a quantitative version of the implicit functions theorem to $F$. See Appendix \ref{sec:implicit} for a statement and a proof adapted to our setting.

\begin{lemma}\label{lem:invertible}
For each $(\nu,h)\in\Omega$, $D_hF$ is invertible iff $ \nu\int \alpha\Theta(\nu, h)\neq 1$. Moreover, there exist $\bar C>1$ such that, setting $\delta=|\nu-\nu_0|+\|h-h_0\|_{\cB_0}$, we have
\[
\begin{split}
&\left\|\Id-D_hF(\nu_0, h_0)^{-1}D_hF(\nu, h)\right\|_{\cB_1}\leq \bar C\delta\ln\delta^{-1}\left\{\left|1-\nu_0\int \alpha\Theta(\nu_0,h_0)\right|^{-1}+1\right\}\\
&\left\|D_hF(\nu_0, h_0)^{-1}D_\nu F(\nu, h)\right\|_{\cB_1}\leq \bar C\left\{\left|1-\nu_0\int \alpha\Theta(\nu_0,h_0)\right|^{-1}+1\right\}.
\end{split}
\]
\end{lemma}
\begin{proof}
If $ \nu\int \alpha\Theta(\nu, h)= 1$, then, using the explicit formula in Lemma \ref{lem:der}, we have $D_hF\Theta(\nu, h)=0$, hence one implication is trivial. Let us investigate the other.
 Note that the tangent space of the domain of $F$ is $\bV\times \bR$. Let $h\in\bW$, and $\psi\in \bV$, we want to find $\phi\in\bV$ such that
\begin{equation}\label{eq:solutions}
D_hF(\phi)=\phi-\nu \Theta(\nu,h)\int \alpha\phi=\psi.
\end{equation}
The above implies that $\phi-\psi=b\Theta(\nu,h)$, for some $b\in\bR$, which in turn implies
\[
b\Theta(\nu,h)=\nu\Theta(\nu,h)\int \alpha(b\Theta(\nu,h)+\psi)
\]
that is
\[
b\left( 1-\nu\int \alpha\Theta(\nu,h)\right)=\nu \int \alpha \psi
\]
and
 \[
 b=\nu \left( 1-\nu\int \alpha\Theta(\nu,h)\right)^{-1}\int\alpha \psi,
 \]
which yields the following unique solution of \eqref{eq:solutions}
\[
\phi=\frac{\nu\Theta(\nu,h)}{1-\nu\int \alpha\Theta(\nu,h)}\int \alpha\psi+\psi.
\] 
Note that, by assumptions (A\ref{ass:6}) and (A\ref{ass:4}) there exists $C>0$ such that, 
for all $\psi\in\bV$, $\|\phi\|_{\cB_1}\leq C\|\psi\|_{\cB_1}$.
Accordingly, $D_{h}F$ is invertible and
\begin{equation}\label{eq:inverse}
(D_hF)^{-1}\psi =\frac{\nu\Theta(\nu,h)}{1-\nu\int \alpha\Theta(\nu,h)}\int \alpha\psi+\psi.
\end{equation}
Thus, by Lemma \ref{lem:der} and equation \eqref{eq:inverse} we can write
\[
\begin{split}
&D_hF(\nu_0, h_0)^{-1}D_hF(\nu, h)\phi=\frac{\nu_0\Theta(\nu_0, h_0)}{1-\nu_0\int \alpha\Theta(\nu_0, h_0)}\\
&\times\int \alpha\left(\phi-\nu \Theta(\nu,h)\int \alpha\phi\right)+\phi-\nu \Theta(\nu,h)\int \alpha\phi\\
&=\frac{1-\nu\int \alpha\Theta(\nu, h)}{1-\nu_0\int \alpha\Theta(\nu_0, h_0)}\nu_0\Theta(\nu_0, h_0)\int \alpha\phi+\phi-\nu \Theta(\nu,h)\int \alpha\phi\\
&=\frac{\left[\nu_0\int \alpha\Theta(\nu_0, h_0)-\nu\int \alpha\Theta(\nu,h)\right]}{1-\nu_0\int \alpha\Theta(\nu_0,h_0)}\nu_0\Theta(\nu_0, h_0)\int \alpha\phi\\
&\phantom{=}+(\nu_0\Theta(\nu_0, h_0)-\nu\Theta(\nu,h))\int \alpha\phi+\phi.
\end{split}
\]
On the other hand, recalling \eqref{eq:Hderivh} and \eqref{eq:Hderivnu},
\begin{equation}\label{eq:deltaTheta}
\begin{split}
&\left\|\Theta(\nu_0, h_0)-\Theta(\nu,h)\right\|_{\cB_1}\leq\\
&\leq\Big\|\left[(\Id-\cL_{T_{\nu_0,h_0}})^{-1}-(\Id-\cL_{T_{\nu,h}})^{-1}\right] \divv (\beta H(\nu_0,h_{ 0}))\|_{\cB_1}\\
&+\|(\Id-\cL_{T_{\nu,h}})^{-1}\divv\left[\beta H(\nu_0,h_{ 0})-\beta H(\nu,h)\right]\Big\|_{\cB_1}\\
&\leq C_0 \|\beta\|_{\cC^{\bar r}}\Big[\|H(\nu_0,h_{ 0})\|_{\cB_3}\left(|\nu-\nu_0|+\|h-h_0\|_{\cB_0}\right)\\
&\phantom{\leq\;}+\|H(\nu_0,h_{ 0})-H(\nu,h)\|_{\cB_2}\Big].
\end{split}
\end{equation}
Hence, the first inequality of the Lemma follows by Lemma \ref{lem:H-prop}.

The second inequality follows by equations \eqref{eq:inverse} and Lemma \ref{lem:der},
\begin{equation}\label{eq:der-h-prelim}
\begin{split}
D_hF(\nu_0, h_0)^{-1}D_\nu F(\nu, h)=& 
\left[\frac{\nu_0\Theta(\nu_0,h_0)\int \alpha \Theta(\nu,h)}{1-\nu_0\int\alpha\Theta(\nu_0,h_0)}
+\Theta(\nu,h)\right]\\
&\times\int \alpha h
\end{split}
\end{equation}
from which the lemma follows.
\end{proof}
\subsection{Local existence of invariant measures}\label{sec:inv_meas}\ \\
To continue, we need a few definitions. Let $\Omega_1$ be the space $[0,\nu_*]\times\bW$, where $\bW$ is endowed with the $\cB_1$ topology. Note that Lemma \ref{lem:der} implies a fortiori that $F\in\cC^1(\Omega_1,\cB_1)$.
In the following we will consider $F\in\cC^1(\Omega_1,\cB_1)$ unless otherwise stated.
Our first result is the following.
\begin{lemma}\label{lem:noturn} 
If $\nu_0\in [0,\nu_*)$, $h_0\in\bW$, $F(\nu_0,h_0)=0$ and $\nu_0\int \alpha\Theta(\nu_0, h_0)\neq~1$, then, setting
\[
\begin{split}
&\delta = \frac{1}{4 \bar C^2\left\{\left|1-\nu_0\int \alpha\Theta(\nu_0,h_0)\right|^{-1}+1\right\}^2}\\
&\delta_1 = \frac{1}{8 \bar C^3\left\{\left|1-\nu_0\int \alpha\Theta(\nu_0,h_0)\right|^{-1}+1\right\}^3},
\end{split}
\]
there exists $\bh\in\cC^1((\nu_0-\delta_1,\nu_0+\delta_1)\cap[0,\nu_*), \bW)$, $\bh (\nu_0)=h_0$, such that all the solutions of $F(\nu, h)=0$ in the set 
\[
\{(\nu,h)\in\bR\times \bW\;:\;  |\nu-\nu_0|\leq \delta_1, \|h-h_0\|_{ \cB_1}\leq \delta\}
\]
belong to the set  $\{\bh(s)\}_{s\in (-\delta_1,\delta_1)}$. Moreover,
\[
\bh'(s)=-\left[\frac{s\Theta(s,\bh(s))\int \alpha \Theta(s,\bh(s))}{1-s\int\alpha\Theta(s,\bh(s))}
+\Theta(s,\bh(s))\right]\int \alpha \bh(s).
\]
\end{lemma}
\begin{proof}
By Lemma \ref{lem:invertible}, if $\nu_0\int \alpha\Theta(\nu_0, h_0)\neq 1$, we can apply Theorem \ref{thm:implicit-func} whose hypotheses are satisfied with the claimed $\delta,\delta_1$.\footnote{We use the inequality $x\ln x^{-1}\leq \sqrt x$.} 

To conclude, note that, by Theorem \ref{thm:implicit-func} and equation \eqref{eq:der-h-prelim},
\[
\bh'(\nu)=-(D_h F)^{-1}D_\nu F=
-\left[\frac{\nu\Theta(\nu,\bh)\int \alpha \Theta(\nu,\bh)}{1-\nu\int\alpha\Theta(\nu,\bh)}
+\Theta(\nu,\bh)\right]\int \alpha \bh.
\]
\end{proof}
We have thus obtained the proof of Theorem \ref{thm:implicit} in the invertible case with $(G(\tau), \bh(\tau))=(\tau,\bh(\tau))$.\\
It remains to analyze what happens if  $\nu_{0}\int\alpha\Theta(h_0,\nu_{0})= 1$. 
To this end, it is convenient to define $Z\in\cB_1'$ by
\begin{equation}\label{eq:Zeta}
Z(\phi)=\int\alpha\phi.
\end{equation}
Note that, for all $\phi\in\bV$, using Lemma \ref{lem:der}, we have
\begin{equation}\label{eq:functional}
Z(D_{h_0}F(\phi))=\int \alpha \left[\phi-\nu_0\Theta(\nu_0,h_0)\int \alpha\phi\right]=0.
\end{equation}
Next, we define the change of variables $\Lambda : \bR\times \bW\to \bR\times \bW$
defined as
\begin{equation}\label{eq:coord-change}
\Lambda(\tau,\zeta)=( Z(\zeta),\zeta-\tau \nu_0\Theta(\nu_0,h_0)).
\end{equation}
One can check directly that $\Lambda$ is invertible with inverse given by
\[
\Lambda^{-1}(\nu,h)=(\nu-Z(h), h-(Z(h)-\nu)\nu_0\Theta(\nu_0,h_0)).
\]
Indeed,
\begin{equation*}
\begin{split}
&\Lambda\circ\Lambda^{-1}(\nu,h)\\
&=(Z(h)-(Z(h)-\nu)\nu_0Z(\Theta(\nu_0,h_0)),\\
&\phantom{=\;\;}h-(Z(h)-\nu)\nu_0\Theta(\nu_0,h_0)-(\nu-Z(h))\nu_0\Theta(\nu_0,h_0))=(\nu,h).
\end{split}
\end{equation*}
Let $(\tau,\zeta)=\Lambda^{-1}(\nu,h)$.
Clearly, the equation $F(\nu, h)=0$ is equivalent to
\begin{equation}\label{eq:implicitmod}
\widetilde F(\tau,\zeta):= F\circ\Lambda(\tau,\zeta)=0.
\end{equation}
Let $(\tau_0,\zeta_0)=\Lambda^{-1}(\nu_0,h_0)$.
The point of this change of variable rests in the following result.
\begin{lemma}\label{lem:invertible2}
If $\nu_{0}\int \alpha\Theta(\nu_0, h_0)= 1$, then, under the hypothesis of Theorem \ref{thm:implicit}, $D_\zeta\widetilde F(\tau_0,\zeta_0)$ is invertible. Moreover, there exist $\bar C>1$ such that, setting $\delta=|\nu-\nu_0|+\|h-h_0\|_{\cB_0}$,
\[
\begin{split}
&\left\|\Id-D_\zeta \widetilde F(0, \zeta_0)^{-1}D_\zeta\widetilde F(\tau, \zeta)\right\|_{\cB_1}\leq \bar C\left[1+\frac{1}{\left|\int\alpha h_0\right|}\right]\delta\ln\delta^{-1}\\
&\left\|D_\zeta \widetilde F(0, \zeta_0)^{-1}D_\tau \widetilde F(\tau, \zeta)\right\|_{\cB_1}\leq \bar C \left[1+\frac{1}{\left|\int\alpha h_0\right|}\right]\delta\ln\delta^{-1} .
\end{split}
\]
\end{lemma}
\begin{proof}

First of all note 
\begin{equation}\label{eq:kernel0}
D_{h}F(\Theta(\nu_0,h_0))=0
\end{equation}
so $D_{h}F(\nu_0,h_0)$ has a kernel. In addition, equation \eqref{eq:functional} shows that the range of $D_{h}F(\nu_0,h_0)$ has, at least, codimension one. 
 
Next, we compute the derivative of $\widetilde F$. Recalling  Lemma \ref{lem:der},
\begin{equation}\label{eq:derim}
\begin{split}
D_{\zeta}\widetilde F&=(D_{h} F)\circ \Lambda+(D_{\nu} F)\circ \Lambda\otimes Z\\
&= \Id-\left[Z(h)+\nu\right]\Theta(\nu,h)\otimes Z.
\end{split}
\end{equation}
Let us show that $D_{\zeta}\widetilde F$ is invertible at $(\nu_0,h_0)$. 
For each $\psi\in \bV$ we must consider the equation
\[
\psi=D_{\zeta}\widetilde F(\phi)=\phi-\left[Z(h_0)+\nu_0\right]\Theta(\nu_0,h_0) Z(\phi).
\]
Applying $Z$ to the above equation and arguing as in \eqref{eq:functional}, we have
\[
Z(\phi)=\left\{1-\left[Z(h_0)+\nu_0\right]Z(\Theta(\nu_0,h_0))\right\}^{-1}Z(\psi)=-\nu_0\left\{Z(h_0)\right\}^{-1}Z(\psi).
\]
Note that $Z(h_0)\neq 0$ by hypothesis \eqref{eq:range}. Thus,
\begin{equation}\label{eq:tildeFinv}
\phi=\psi- \frac{\nu_0 \left[Z(h_0)+\nu_0\right]}{Z(h_0)}\Theta(\nu_0,h_0) Z(\psi)=\left[D_\zeta \widetilde F(\nu_0, h_0)\right]^{-1}\hskip-6pt(\psi).
\end{equation}
Accordingly, by the bounded inverse Theorem, $D_{\zeta}\widetilde F$ is invertible. 

To prove the last inequality of the statement of the Lemma, we recall Lemma \ref{lem:der} and we write
\[
\begin{split}
&D_\zeta \widetilde F(0, \zeta_0)^{-1}D_\zeta\widetilde F(\tau, \zeta)(\phi)=\phi-\left[Z(h)+\nu\right]\Theta(\nu,h) Z(\phi)\\
&\phantom{=}- \frac{\nu_0 \left[Z(h_0)+\nu_0\right]}{Z(h_0)}\Theta(\nu_0,h_0) \left\{1-\left[Z(h)+\nu\right]Z(\Theta(\nu,h))\right\} Z(\phi)\\
&=\phi-\left[Z(h)-Z(h_0)+\nu-\nu_0\right]\Theta(\nu,h) Z(\phi)\\
&\phantom{=}
-\left[Z(h_0)+\nu_0\right]\left\{\Theta(\nu,h)-\Theta(\nu_0,h_0)\right\} Z(\phi)\\
&\phantom{=}+\frac{\nu_0 \left[Z(h_0)+\nu_0\right]}{Z(h_0)}\Theta(\nu_0,h_0)\left\{Z(h_0)Z(\Theta(\nu_0,h_0))-Z(h)Z(\Theta(\nu,h))\right\} Z(\phi)\\
&\phantom{=}-\frac{\nu_0 \left[Z(h_0)+\nu_0\right]}{Z(h_0)}\Theta(\nu_0,h_0)\left\{\nu_0 Z(\Theta(\nu_0,h_0))-\nu Z(\Theta(\nu,h))\right\} Z(\phi).
\end{split}
\]
The first inequality of the Lemma follows by \eqref{eq:deltaTheta} and Lemma \ref{lem:H-prop}.
Let us prove the second. Note that
\[
\begin{split}
D_\tau \widetilde F(\tau, \zeta)&=-\nu_0D_h F(\nu,h)\Theta(\nu_0,h_0)\\
&=-\nu_0\Theta(\nu_0,h_0)+\nu_0\nu\Theta(\nu,h)Z(\Theta(\nu_0,h_0))\\
&= \nu\Theta(\nu,h)-\nu_0\Theta(\nu_0,h_0).
\end{split}
\]
Hence, recalling equations \eqref{eq:tildeFinv}, \eqref{eq:deltaTheta} and Lemma \ref{lem:H-prop},
\begin{equation}\label{eq:wFtau}
\begin{split}
D_\zeta\widetilde F(0, \zeta_0)^{-1}&D_\tau \widetilde F(\tau, \zeta)=D_\tau \widetilde F(\tau, \zeta)\\
&-\frac{\nu_0\left[Z(h_0)+\nu_0\right]}{Z(h_0)}\Theta(\nu_0,h_0)Z(D_\tau \widetilde F(\tau, \zeta))\\
=&\nu\Theta(\nu,h)-\nu_0\Theta(\nu_0,h_0)\\
&-\frac{\nu_0\left[Z(h_0)+\nu_0\right]}{Z(h_0)}\Theta(\nu_0,h_0)Z(\nu\Theta(\nu,h)-\nu_0\Theta(\nu_0,h_0)).
\end{split}
\end{equation}
The Lemma follows by using \eqref{eq:deltaTheta} and Lemma \ref{lem:H-prop}.
\end{proof}
\begin{lemma}\label{lem:turn}
If $\nu_0\in (0,\nu_*)$, $h_0\in\bW$, $F(\nu_0,h_0)=0$, $\nu_0\int \alpha\Theta(\nu_0, h_0)=~{\!\!1}$, assumption \eqref{eq:range} holds, and $\int \alpha h_0\not=0$ then there exist $\delta,\delta_1>0$, with $\delta_1\leq \nu_0^{-1}\|\Theta(\nu_0,h_0)\|_{\cB_1}^{-1}\delta$, and a differentiable map $(G,\bh)=:(-\delta_1,\delta_1)\to [0,\nu_*]\times \bW$, $(G(0),\bh (0))=(\nu_0,h_0)$, such that all the solutions of $F(\nu, h)=0$ in the set 
\[
\{(\nu,h)\in\bR\times \bW\;:\;  |\nu-\nu_0|\leq \delta_1, \|h-h_0\|_{ \cB_1}\leq \delta\}
\]
belong to the set  $\{G(\tau),\bh(\tau)\}_{\tau\in(-\delta_1,\delta_1)}$.\\
If $Z(D_h^2|_{(\nu_0,h_0)}F(\Theta(\nu_0,h_0), \Theta(\nu_0,h_0))\neq 0$, then, setting $\nu(\tau)=G(\tau)$,
 \[
 \begin{split}
 &\nu'(\tau_0)=0\\
 &\nu''(\tau_0)=\frac{ \nu_0^3Z\left(D_h^2F(\Theta(\nu_0,h_0),\Theta(\nu_0,h_0))\right)}{\int\alpha h_0}\neq 0.
 \end{split}
 \]
\end{lemma}
\begin{proof}
 We use the coordinates \eqref{eq:coord-change}. Note that, by the definition \eqref{eq:coord-change}, $\Lambda\in\cC^1(\Omega_1,\Omega_1)$, so $\widetilde F\in\cC^1(\Omega_1,\cB_1)$.\footnote{ See the beginning of Section \ref{sec:inv_meas} for the definition of $\Omega_1$.}
Note that there exists $C_0>0$ such that, if $\|\zeta-\zeta_0\|_{\cB_1}+\nu_0\|\Theta(\nu_0,h_0)\|_{\cB_1}|\tau|\leq \delta_0$, then $\|h-h_0\|_{\cB_1}\leq C_0\delta_0$ and $|\nu-\nu_0|\le C_0\delta_0$.
By Lemma \ref{lem:invertible2} we can then choose $\delta_0$ so that 
\[
\left\|\Id-D_\zeta \widetilde F(0, \zeta_0)^{-1}D_\zeta\widetilde F(\tau, \zeta)\right\|_{\cB_1}\leq\frac 12,
\]
and
\[
\left\|D_\zeta \widetilde F(0, \zeta_0)^{-1}D_\tau \widetilde F(\tau, \zeta)\right\|_{\cB_1}\leq\frac12.
\]
Hence, the Implicit Function Theorem \ref{thm:implicit-func} yields a differentiable function $\zeta(\tau)$ such that $\widetilde F(\tau,\zeta(\tau))=0$ for all $|\tau|\leq\frac{\delta_0}{\nu_0\|\Theta(\nu_0,h_0)\|_{\cB_1}}$. The first assertion of the Lemma follows with $(G(\tau),\bh(\tau))=\Lambda(\tau,\zeta(\tau))$. Hence,
\begin{equation}\label{eq:derone}
D_hF (h')+D_\nu F\nu'=D_hF (\zeta'-\nu_0\Theta(\nu_0,h_0))+D_\nu F\nu'=0
\end{equation}
which, applying $Z$ and recalling \eqref{eq:functional} and since assumption \eqref{eq:range} implies $Z(D_{\nu} F)(\nu_0,h_0)\neq 0$, gives $\nu'(\tau_0)=0$. Note that $\nu'(\tau_0)=Z(\zeta'(\tau_0))$, hence 
\[
\begin{split}
0=&D_{h}{ F}  (\zeta'(\tau_0)- \nu_0\Theta(\nu_0,h_0))\\
=& \zeta'(\tau_0)-\nu_0\Theta(\nu_0,h_0)-\nu_0\Theta(\nu_0,h_0)Z( \zeta'- \nu_0\Theta(\nu_0,h_0))\\
=&\zeta'(\tau_0)-\nu_0\Theta(\nu_0,h_0)+\nu_0^{2}\Theta(\nu_0,h_0)Z(\Theta(\nu_0,h_0))=\zeta'(\tau_0).
\end{split}
\]
Accordingly, 
\begin{equation}\label{eq:h-der}
\bh'(\tau_0)=\zeta'(\tau_0)- \nu_0\Theta(\nu_0,h_0)=-\nu_0\Theta(\nu_0,h_0).
\end{equation}
Next, we can differentiate \eqref{eq:derone} considering $F\in\cC^1(\Omega, \cB_0)$ and obtain
\[
D_h^2F(\bh',\bh')+2D^2_{\nu, h}F \bh'\nu'+D_hF\bh''+D_\nu^2F(\nu')^2+D_\nu F\nu''=0.
\]
Setting $\tau=\tau_0$ and applying $Z$ yields
\[
\begin{split}
0&=\nu_0^2Z\left(D_h^2F(\Theta(\nu_0,h_0),\Theta(\nu_0,h_0))\right)+Z(D_\nu F(\nu_0,h_0))\nu''(\tau_0)\\
&=\nu_0^2Z\left(D_h^2F(\Theta(\nu_0,h_0),\Theta(\nu_0,h_0))\right)-\nu_0^{-1}\int\alpha h_0\nu''(\tau_0),
\end{split}
\]
which, by hypothesis, implies $\nu''(\tau_0)\neq 0$. 
\end{proof}
\begin{proof}[\bf\em Proof of Theorem \ref{thm:implicit}]
The existence of the functions $(G,\bh)$ follows immediately from Lemmata \ref{lem:noturn} and \ref{lem:turn}.
If $\nu_0\int \alpha\Theta(\nu_0, \bh(\nu_0))= 1$, then by Lemma \ref{lem:turn} we have $G(0)=\nu_0$, $G'(0)=0$ and $G''(0)\neq 0$, hence for each $\tau\in (-\delta,\delta)$, with $\delta$ small enough, we have that there exists $\tau_1<0$ such that $\nu=G(\tau)=G(\tau_1)$ while, by \eqref{eq:h-der},
\[
\bh(\tau)-\bh(\tau_1)=\int_{\tau_1}^\tau\bh'(s)ds=(\tau_1-\tau) \Theta(\nu_0,h_0)+\cO((\tau-\tau_1)^2),
\]
where we have used Lemma \ref{lem:invertible2}.
Thus $\tilde\cL_\nu h=h$ has at least the two solutions $\{\bh(\tau),\bh(\tau_1)\}$.
\end{proof}
\subsection{Global existence of invariant measures}\label{sec:invariant_global}\ \\
This section is devoted to the proof of Theorem \ref{thm:implicitg}, which provides a criterion that implies the existence of a $\nu$ for which there exists at least three invariant measures.

To start with, we prove that the solutions in Theorem \ref{thm:implicit} can be extended.
\begin{proposition}\label{prop:endinv} Let $\nu_0\in (0,\nu_*)$ and $h_0\in\bW$ such that $\tilde\cL_{\nu_0}h_0=h_0$ and $\nu_0\int\alpha\Theta(\nu_0,h_0)\neq 1$. Then there exists an interval $[\nu_1, \nu_2]\subset [0,\nu_*]$, with $\nu_0\in (\nu_1, \nu_2)$ and $\bh\in\cC^0([\nu_1,\nu_2],\cB_1)\cap \cC^1((\nu_1,\nu_2),\cB_1)$ such that $\cL_{T_{\nu,\bh(\nu)}}\bh(\nu)=\bh(\nu)$. Moreover,  either $\nu_1=0$ or $\nu_1\int \alpha\Theta(\nu_1,\bh(\nu_1))=1$ and  either $\nu_2=\nu_*$ or $\nu_2\int \alpha\Theta(\nu_2,\bh(\nu_2))=1$. 
\end{proposition}
\begin{proof}
We discuss the extension to $\nu_1$, the other being identical. Theorem \ref{lem:noturn} implies that there exist $(a,b)\ni \nu_0$ and $\bh\in\cC^1((a,b),\cB_1)$ such that $\cL_{T_{\nu,\bh(\nu)}}\bh(\nu)=\bh(\nu)$. Lemma \ref{lem:H-prop} implies that $\bh(a,b)$ is relatively compact in $\cB_1$. Let $h_*$ be an accumulation point when $\nu\to a$. Then, by Lemma \ref{lem:regularity},
\[
h_*=\lim_{\nu_j\to a}\bh_{\nu_j}=\lim_{\nu_j\to a}\cL_{T_{\nu_j,\bh(\nu_j)}}\bh_{\nu_j}=\cL_{T_{a,h_*}}h_*.
\]
We can then set $\bh(a)=h_*$ and, if $a\int \alpha\Theta(a,\bh(a))= 1$ or $a=0$ we are done, otherwise we can apply again Theorem \ref{lem:noturn} with $\nu=a$.
Clearly this process can stop only at $0$ or at $\nu_1>0$ if $\nu_1\int \alpha\Theta(\nu_1,\bh(\nu_1))=1$.
This provides a continuous curve, yet the last formula of Lemma \ref{lem:noturn} implies that $\bh\in\cC^1((\nu_1,\nu_2),\cB_1)$.
\end{proof}

We are now ready to conclude the argument.
\begin{proof}[\bfseries Proof of Theorem \ref{thm:implicitg} ]
For $\nu=0$, the map is simply $T$, and by Lemma \ref{lem:spectra}, it has a unique invariant measure $h\in \cB_1$.  By Lemma \ref{lem:der} $D_hF=\Id$, thus we can apply Theorem \ref{thm:implicit} and we can extend the solution in an interval $[0,\nu_0)$. 
Then Proposition \ref{prop:endinv} implies that the solution can be extended to all the interval $[0,\nu_*]$ or we lose the invertibility of $D_hF$. In the latter case, we can still extend the solution using Theorem \ref{thm:implicit}. In this way we can define a function $(G,\bh):[0, a)\to [0,\nu_*)\times \bW$, $a\in\bR\cup\infty$, $G(0)=0$,with the wanted properties. Note that $G((0,a))\cap\{0\}=\emptyset$, since we know that $F(0,h)=0$ has the unique solution $\bh(0)$.

Suppose that there exists $\nu_1\in [0,\nu_*]$ such that $\sharp G^{-1}(\nu_1)=\infty$, then there exists $\{a_i\}_{i\in\bN}$ such that $\widetilde\cL_{\nu_1}\bh(a_i)=\bh(a_i)$. By Lemma \ref{lem:H-prop}, $\cH:=\{\bh(a_i)\}$ belong to a bounded set $\bW_2:= \{h\in\cB_2\cap\bW\;:\; \|h\|_{\cB_2}\leq K_\star\}$, so they are relative compact in $\cB_1$ by Assumption (A\ref{ass:3}). Let $h_*\in\cB_1$ be an accumulation point, then, by continuity $\widetilde\cL_{\nu_1}h_*=h_*$. It follows that the $\cB_1$ closure of $\cH$ consists of fixed points on $\widetilde\cL_{\nu_1}$. Let $B_\delta(h)=\{g\in\bW\;:\;\|h-g\|_{\cB_1} \leq \delta\}$. Then, for each $h\in\bW_2$, if  $h\not \in\overline{\cH}$, there exists $\delta_h>0$ such that $B_{\delta_h}(h)\cap\overline{\cH}=\emptyset$, while if $h\in  \overline{\cH}$ then Theorem \ref{thm:implicit} implies that there exists $\delta_h$ such that $h$ is the only solution of $\widetilde\cL_{\nu_1}g=g$ in $B_{\delta_h}(h)$. Since $\{B_{\delta_h}(h)\}_{h\in\bW_2}$ is a covering of $\bW_2$, which is relatively compact in $\cB_1$, by Assumption (A\ref{ass:3}), we can extract a finite subcover, and this immediately implies that $\sharp\cH<\infty$.

It follows that the curve we have constructed cannot go back and forth infinitely many times in the interval $(0,\nu^*)$, and it cannot go back to $\nu=0$, hence it must eventually reach $\nu^*$.

To conclude, if $\widetilde\cL_{\bar\nu} h=h$ for some $\bar \nu<\nu^*$, and $D_hF$ is not invertible, then there are two possibilities: either $h$ belongs to the curve we constructed starting from zero or not. In the first case, we have that the curve is described by a function $(G,\bh)$, $G(0)=0$, and for some $\tau_0$, we have $G(\tau_0)=\nu$ and $\bh(\tau_0)=h$. We can assume without loss of generality that $\tau_0$ is the first time $G'=0$. Since the second derivative is non zero, it follows that for $\tau>\tau_0$ $G$ is decreasing. Since it cannot go back to zero it must have another point (we call them turning points) in which $G'=0$ and start to increase again. Since it can have only finitely many turning points, it must eventually go back to $\tau_0$, and that means that there exists $\epsilon>0$ such that for $\tau\in (\tau_0-\epsilon,\tau_0)$ we have $\sharp G^{-1}(\tau)\geq 3$.
If instead $G'(\tau)\neq 0$ for all $\tau$, then the curve starting from zero has no turning point, and in the interval $[0,\nu_*]$ it provides an invariant measure for each value of $\tau$. It follows that $h$ does not belong to such a curve. We can then apply Theorem \ref{thm:implicit} and Proposition \ref{prop:endinv}  to create a curve of solution starting from $h$. Such a curve cannot end at zero hence it must end, in both directions, at $\nu_*$. If we call $(G_1,\bh_1)$ such a curve, near $\nu_*$ it must have at least two branches. Hence, together with the curve $(G,\bh)$, provides again at least three invariant measures.
\end{proof}
\section{Local stability and physical measures}\label{sec:physical}
Here we study the stability of the invariant measures for a fixed $\nu$. Since if $\widetilde \cL_\nu h=h$, then $h\in\cB_3$ (see Lemma \ref{lem:H-prop}) it will be helpful to study $\widetilde \cL_{\nu}$ as a function from $\cK:=\{h\in\cB_3 \cap (\cC^0)'\;:\;h\geq 0, \int h=1\} \subset \cB_1$ to $\cB_1$. Note that $\cB_3 \cap (\cC^0)'$  is a space of measures. Hence $\cK$ is a space of probability measures. With the $\cB_1$ topology, $\cK$ is an affine space modeled over a normed space (but not a Banach space). Moreover, 
$\widetilde \cL_\nu\cK\subset \cK$, hence we can, and will, restrict our study to $\cK$. Since we want to talk about derivatives of functions on $\cK$, we need to define a tangent space. Given the nonstandard topology on $\cK$, it is not obvious what to choose as a tangent space. We find it convenient to define $T\cK$ (the tangent space to $\cK$) as $\{\phi\in\cB_2\cap (\cC^0)'\;:\; \int \phi=0\} $ equipped with the $\cB_1$ norm.
\begin{lemma}\label{lem:Gat}
$\widetilde\cL_\nu\in\cC^1(\cK,\cK)$ with G\^ateaux differential, at $h\in \cK$, $\cD_h(\phi)\in L(T\cK,T\cK)$ defined by\footnote{We refer to \cite{DM13} for the theory of differentiability on $\cK$. To be precise, \cite[Section 3.2]{DM13} deals with functions between normed vector spaces, not affine spaces modeled on normed vector spaces, yet the theory applies to the latter case essentially verbatim.}
\begin{equation}\label{eq:LT_der}
\cD_h(\phi):={D}\widetilde \cL_{\nu} |_{h}(\phi)=\cL_{T_{\nu,h}}\phi-\nu\divv \cL_{ T_{\nu,h}}\beta h\int\alpha\phi.
\end{equation}
\end{lemma}
\begin{proof}
For each $\phi\in T\cK$, and $t\in\bR$,
\[
\widetilde \cL_{\nu}(h+t\phi)-\widetilde \cL_{\nu}(h)=(\cL_{T_{\nu,h+t\phi}}-\cL_{T_{\nu,h}})(h)+t\cL_{T_{\nu,h+t\phi}}(\phi).
\]
We start by computing the G\^ateaux derivative.
By Lemma \ref{lem:regularity} we have
\[
\lim_{t\to 0}\frac{\|\cL_{T_{\nu,h+t\phi}}-\cL_{T_{\nu,h}})(h)+t \nu\divv \cL_{ T_{\nu,h}}\beta h\int\alpha\phi\|_{\cB_1}}{|t|}=0.
\]
The results follows since, by Lemma \ref{lem:regularity},
\[
\lim_{t\to 0} \|\cL_{T_{\nu,h+t\phi}}(\phi)-\cL_{T_{\nu,h}}(\phi)\|_{\cB_1}=0.
\]
It follows that $\cD_h(\phi)$ is the G\^ateaux differential. Since $\cD_h\in L(T\cK,T\cK)$ the lemma follows.\end{proof}
Since the chain rule does not apply to the G\^ateaux Differential, we state a poor man substitute of the chain rule for the case at hand.
\begin{lemma}\label{lem:chain-rule}
For each  $g\in\cK$, and $\phi\in T\cK$, there exists a constant $H, J>1$ such that for each $\phi\in T\cK$ and  $n\geq \bar n$ we have
\[
\begin{split}
&\|\widetilde\cL_{\nu}^n(g+\phi)-\widetilde\cL_{\nu}^n(g)-\cD_{g_{k-1}}\cdots \cD_{g}(\phi)\|_{\cB_1}\leq H^n\|\phi\|_{\cB_2}\|\phi\|_{\cB_0}\\
&\|\widetilde\cL_{\nu}^n(g+\phi)-\widetilde\cL_{\nu}^n(g)\|_{\cB_2}\leq J\Big(H^n\|\phi\|_{\cB_2}\|\phi\|_{\cB_0} \\
&\phantom{\|\widetilde\cL_{\nu}^n(g+\phi)-\widetilde\cL_{\nu}^n(g)\|_{\cB_2}\leq}
+\sum_{j=0}^{k-1}\vartheta^{k-j}\|\cD_{g_{j-1}}\cdots \cD_{g}\phi\|_{\cB_1}+\vartheta^k\|\phi\|_{\cB_2}\Big).
\end{split}
\]
\end{lemma}

\begin{proof}
For each $k\in\bN$ let $g_k=\widetilde\cL_\nu^k(g)$ and $\phi_k=\widetilde\cL_\nu^k(g+\phi)-g_k$.  With this definition we get $\widetilde\cL_\nu^k(g+\phi)= \widetilde\cL_{\nu}(g_{k-1}+\phi_{k-1})$. 
We start with rough estimates in the $\cB_i$ norms: by Assumption (A\ref{ass:6}) we have
\[
\|g_k\|_{\cB_3}\leq C_*\vartheta^k\|g\|_{\cB_3}+K\|g\|_{\cB_2}\leq (C_*+K)\|g\|_{\cB_3}=:R,
\]
and, for $i\in\{0,1,2\}$, by Lemma \ref{lem:regularity} it follows 
\[
\begin{split}
\|\phi_k\|_{\cB_i}&\leq \|\widetilde\cL_{\nu}(g_{k-1}+\phi_{k-1})-\widetilde\cL_{\nu}(g_{k-1})\|_{\cB_i}\\
&=\|\cL_{\nu,T_{\nu,g_{k-1}+\phi_{k-1}}}\phi_{k-1}\|_{\cB_i}+\|\cL_{\nu,T_{\nu,g_{k-1}+\phi_{k-1}}}g_{k-1}-\cL_{\nu,T_{\nu,h}}g_{k-1}\|_{\cB_i}\\
&\leq \|\cL_{\nu,T_{\nu,g_{k-1}+\phi_{k-1}}}\|_{L(\cB_i,\cB_i)}\|\phi_{k-1}\|_{\cB_{i}}+C_\star|\nu|\|\beta\|_{\cC^{\bar r}}\|\alpha\|_{\cC^{\bar r}}
\|g_{k-1}\|_{\cB_{i+1}}\|\phi_{k-1}\|_{\cB_{0}}\\
&\leq C\|\phi_{k-1}\|_{\cB_{i}}\leq C^k\|\phi\|_{\cB_{i}}.
\end{split}
\]
for some constant $C>0$ depending only on $T,\nu,\beta,\alpha$ and $R$.
Next, Lemma \ref{lem:regularity} implies (in the $\cB_1$ topology)
\[
\begin{split}
\phi_k&=\widetilde\cL_\nu(g_{k-1}+\phi_{k-1})-\widetilde\cL_\nu(g_{k-1})\\
&=\cL_{T_{\nu,g_{k-1}+\phi_{k-1}}}(g_{k-1})-\cL_{T_{\nu,g_{k-1}}}(g_{k-1})+\nu\divv \cL_{ T_{\nu,g_{k-1}}}\beta g_{k-1}\int\alpha\phi_{k-1}\\
&\phantom{=}
+\cL_{T_{\nu,g_{k-1}}}(\phi_{k-1})-\nu\divv \cL_{ T_{\nu,g_{k-1}}}\beta g_{k-1} \int\alpha\phi_{k-1}\\
&\phantom{=}
+\cL_{T_{\nu,g_{k-1}+\phi_{k-1}}}(\phi_{k-1})-\cL_{T_{\nu,g_{k-1}}}(\phi_{k-1})\\
&=\cD_{g_{k-1}}(\phi_{k-1})+\cL_{T_{\nu,g_{k-1}+\phi_{k-1}}}(\phi_{k-1})-\cL_{T_{\nu,g_{k-1}}}(\phi_{k-1})
+\cO(\|\phi_{k-1}\|_{\cB_0}^2R)\\
&=\cD_{g_{k-1}}(\phi_{k-1})+\cO(\|\phi_{k-1}\|_{\cB_2}\|\phi_{k-1}\|_{\cB_0}+\|\phi_{k-1}\|_{\cB_0}^2R)\\
&=\cD_{g_{k-1}}\cdots \cD_{g}(\phi_0)\\
&\phantom{=}
+\sum_{j=0}^{k-1}\cO\left(\|\cD_{g_{j-1}}\cdots \cD_{g}\|_{L(\cB_1,\cB_1)}\left[\|\phi_{k-j}\|_{\cB_2}\|\phi_{k-j}\|_{\cB_0})+
\|\phi_{k-j}\|_{\cB_0}^2R\right]\right).
\end{split}
\]
Hence, by our previous rough estimate, there exists $H>1$ such that
\[
\|\phi_k-\cD_{g_{k-1}}\cdots \cD_{g}(\phi)\|_{\cB_1}\leq H^k\|\phi\|_{\cB_2}\|\phi\|_{\cB_0}.
\]
This proves the first part of the Lemma.
By  Assumptions (A\ref{ass:4}), (A\ref{ass:6}), (A\ref{ass:1}) and Lemma \ref{lem:regularity}, we have (in the $\cB_2$ topology)
\[
\begin{split}
\phi_k&=\widetilde\cL_{\nu}^k(g+\phi)-g_k=\cL_{T_{\nu,g_{k-1}+\phi_{k-1}}}\cdots \cL_{T_{\nu,g+\phi_0}}(g+\phi)-g_k\\
&=\cL_{T_{\nu,g_{k-1}+\phi_{k-1}}}\cdots \cL_{T_{\nu,g+\phi}}(g)-\cL_{T_{\nu,g_{k-1}}}\cdots \cL_{T_{\nu,g}}(g)+\cO(\vartheta^k\|\phi\|_{\cB_2})\\
&=\sum_{j=0}^{k-1}\cL_{T_{\nu,g_{k-1}+\phi_{k-1}}}\cdots  \cL_{T_{\nu,g_{j+1}+\phi_{j+1}}}\left[\cL_{T_{\nu,g_j+\phi_{j}}}g_j-g_{j+1}\right]+\cO(\vartheta^k\|\phi\|_{\cB_2})\\
&=\sum_{j=0}^{k-1}\cO\left(\nu\vartheta^{k-j}R\int \alpha \phi_j\right)+\cO(\vartheta^k\|\phi\|_{\cB_2})\\
&=\cO\left(R H^k\|\phi\|_{\cB_2}\|\phi\|_{\cB_0} +\sum_{j=0}^{k-1}\vartheta^{k-j}\|\cD_{g_{j-1}}\cdots \cD_{g}\phi\|_{\cB_0}+\vartheta^k\|\phi\|_{\cB_2}\right),
\end{split}
\]
where in the last line, we have used the first part of the lemma.
From the above the lemma follows.
\end{proof}
Note that, since $T\cK$ is equipped with the $\cB_1$ topology, Lemma \ref{lem:Gat} and the density assumption \eqref{eq:embedding} imply that $\cD_h$ can be extended to $\cB_1$. With a slight abuse of notation, we will call such an extension $\cD_h$ as well, so $\cD_h\in L(\cB_1,\cB_1)$, moreover $\cD_h(\bV)\subset \bV$. Moreover, since $\cD_h$ is a compact perturbation of $\cL_{T_{\nu,h}}\in L(\cB_1,\cB_1)$, 
$\sigma_{\operatorname{ess}}(\cD_h)=\sigma_{\operatorname{ess}}(\cL_{T_{\nu,h}})$. We can then limit ourselves to the study of the discrete (non-essential) spectrum. To this end, recall the definitions of $\Theta(z)$ and $\Xi(z)$ in \eqref{eq:Thetaz}, $\sigma(\cL_{T_{\nu,h}})\setminus\{1\}=\sigma(Q_{T_{\nu,h}})$ and notice that, for $z\not\in\sigma(Q_{T_{\nu,h}})$, we can write
\begin{equation}\label{eq:tetazeta}
\Theta(z)=-\nu (z\Id-Q_{T_{\nu,h}})^{-1}\divv \cL_{ T_{\nu,h}}\beta h.
\end{equation}
\begin{proposition}\label{prop:stability}
Assume that $\Xi(1)\neq 1$, and
\[
\int \alpha h\neq 0.
\]
Then $z\in\sigma(\cD_h)\setminus \sigma(\cL_{T_{\nu,h}})$ iff $\Xi(z)=1$. In addition, $1\in\sigma(\cD_h)$. Moreover, if $z$ is an eigenvalue with $|z|\geq 1$, then it is associated with a unique eigenvector, and the corresponding eigenvector belongs to $\cB_2$. Moreover, if $\Xi(1)>1$, then there exists $z_1\in \sigma(\cD_h)$ such that $|z_1|>1$ (at most finitely many); while if $\sup_\theta |\Xi(e^{i\theta})|<1$, then there exists $\kappa<1$ such that $\sigma(\cD_h)\setminus\{1\}\subset \{z\in\bC\;:\; |z|<\kappa\}$.
\end{proposition}
\begin{proof}
We look for the $\phi\in\cB_1\setminus \{0\}$ such that $\cD_h(\phi)=z\phi$ where $z\in\bC\setminus \sigma_{\operatorname{ess}}(\cL_{T_{\nu,h}})$. 
Then, using \eqref{eq:LT_der},
\begin{equation}\label{eq:stability1}
-\nu\divv \cL_{ T_{\nu,h}}\beta h\int\alpha\phi=z\phi-\cL_{T_{\nu,h}}\phi.
\end{equation}
If $\divv \cL_{ T_{\nu,h}}\beta h\equiv 0$, then $\cD_h=\cL_{T_{\nu,h}}$ and the Lemma is trivially true. We can then consider only the case $\divv \cL_{ T_{\nu,h}}\beta h\not\equiv 0$.
If $\int \phi\neq 0$, then we can assume, without loss of generality, that $\int\phi=1$ and integrating \eqref{eq:stability1} yields
\begin{equation}\label{eq:eigenone}
\begin{split}
&z=1\\
&(\Id-\cL_{T_{\nu,h}})\phi=-\nu\divv \cL_{ T_{\nu,h}}\beta h\int\alpha\phi.
\end{split}
\end{equation}
If we set $\phi=h+\hat\phi$, $\hat\phi\in\bV$, then by \eqref{eq:eigenone}
\[
\hat\phi=\Theta(1)\int\alpha\phi
\]
which, multiplying by $\alpha$ and integrating, has a solution only if $\Xi(1)\neq 1$ and yields
\[
\begin{split}
&\int\alpha\hat\phi=[1-\Xi(1)]^{-1}\Xi(1)\int \alpha h\\
&\int\alpha\phi=[1-\Xi(1)]^{-1}\int \alpha h.
\end{split}
\]
The above implies that 
\begin{equation}\label{eq:one_eigen}
\phi=h+\Theta(1)[1-\Xi(1)]^{-1}\int \alpha h
\end{equation} 
is and eigenvector, and hence $1\in\sigma(\cD_h)$.
We are then left with the case $\phi\in\bV$.\footnote{Indeed, $\cD_h\bV\subset \bV$.} If $z\not\in\sigma(\cL_{T_{\nu,h}})$, then
equation \eqref{eq:stability1} can be rewritten as
\begin{equation}\label{eq:phistab2}
\phi=\Theta(z)\int\alpha\phi
\end{equation}
which has a nontrivial solution only if $\int\alpha\phi\neq 0$, so we can assume $\int\alpha \phi=1$.
Note that $\Theta(z)\in\cB_2$ since $h\in\cB_3$.
Multiplying equation \eqref{eq:phistab2} by $\alpha$ and integrating yields
\begin{equation}\label{eq:thetaxi}
\Xi(z)=1.
\end{equation}
Conversely, if $\Xi(z)=1$, then $\phi=\Theta(z)$ is a solution of \eqref{eq:phistab2}, and, by \eqref{eq:LT_der},
\begin{equation}\label{eq:eigenvector}
\cD_h\Theta(z)=\cL_{T_{\nu,h}}\Theta(z)+(z\Id-\cL_{T_{\nu,h}})\Theta(z)=z\Theta(z).
\end{equation}
It follows that, for $z\neq 1$, $z\in\sigma(\cD_h)\setminus \sigma(\cL_{T_{\nu,h}})$ iff $\Xi(z)=1$.

Note that, by the equation \eqref{eq:Thetaz}, for $z\in\bR$, we have $\Xi(z)\in\bR$, moreover $\lim_{z\to\infty}\Xi(z)=0$. Accordingly, if $\Xi(1)>1$, there must be a $z_1\in(1,\infty)$ such that $\Xi(z_1)=1$ and $z_1\in\sigma(\cD_h)$, moreover \eqref{eq:phistab2} shows that there is a unique eigenvalue, although there could be a finite-dimensional Jordan block. Also, since $\Xi$ is analytic, there can be at most finitely many such values. On the other hand, if  $\sup_\theta |\Xi(e^{i\theta})|<1$, then, since $\Xi$ is holomorphic outside the unit disk and goes to zero at infinity, we have $\sup_{|z|\geq 1} |\Xi(z)|<1$. Accordingly, all the eventual solutions of \eqref{eq:thetaxi} must be strictly inside the unit disk.
\end{proof}
Before proving our main result, we need one more fact.
\begin{lemma} \label{lem:weak}
$\widetilde \cL_{\nu}:\cK\to\cK$ is weakly continuous.\footnote{ That is, here we view $\cK$ as a subset of the space of probability measures.}
\end{lemma}
\begin{proof}
If $g_n\Longrightarrow g$, weakly, then, for all $\phi\in\cC^1$ we have\footnote{Recall that by the weak convergence of the $g_n$ and the definition \eqref{eq:coupledmap} follows $\lim_{n\to\infty}\|T_{\nu,g_n}-T_{\nu,g}\|_{\cC^1}=0$.}
\[
\lim_{n\to\infty}\int\widetilde \cL_{\nu}(g_n)\phi-\widetilde \cL_{\nu}(g)\phi=\lim_{n\to\infty}\int \phi\circ T_{\nu,g_n} g_n-\phi\circ T_{\nu,g}g=0.
\]
The lemma follows since, for each $\tilde g\in\cK$,
\[
\left|\int\widetilde \cL_{\nu} (\tilde g)\phi\right|\leq \|\phi\|_\infty,
\]
so the result for all continuous functions can be obtained by approximation.
\end{proof}
\begin{proof}[\bf \em Proof of Theorem \ref{lem:physical}]\ \\
 If $h$ is visible, let $g\in\cC^\infty$ be such that $\widetilde \cL_{\nu}^n(g)$ converges weakly to $h$. Then, by  Assumptions (A\ref{ass:6}), (A\ref{ass:1}) for all $n\in\bN$,
\[
\|\widetilde \cL_{\nu}^ng\|_{\cB_3}\leq C_*\vartheta^n\|g\|_{\cB_3}+K\|g\|_{\cB_2}\leq (C_*+K)\|g\|_{\cC^{\bar r}}=:C_g.
\]
Since $\{\vf\in\cB_3\;:\; \|\vf\|_{\cB_3}\leq C_g\}$ is relatively compact in $\cB_2$ by Assumption  (A\ref{ass:3}), there exists a sequence $n_j$ such that $\widetilde \cL_{\nu}^{n_j}(g)$ converges to some $h_*\in\cB_2$. 
Thus, for all $\phi\in\cC^{\bar r}$ we have, recalling Assumption (A\ref{ass:1}),
\[
\int h_*\phi=\lim_{j\to\infty}\int \widetilde \cL_{\nu}^{n_j}(g)\phi=\int \phi h.
\]
By the embedding properties of the Banach spaces, $h=h_*\in\cB_2$. Hence, $\cL_{\nu,h}h=h$, so by Lemma \ref{lem:spectra} $h\in\cB_3$. The formula for the derivative is then just \eqref{eq:LT_der}, and, 
recalling \eqref{eq:one_eigen},
\begin{equation}\label{eq:Rspec}
\begin{split}
\cD_h(\phi)&=\left\{h+  \Theta(1)[1-\Xi(1)]^{-1}\int\alpha h \right\}\int \phi +\cR(\phi)\\
&=\Pi_0\phi+\cR(\phi)
\end{split}
\end{equation}
where $\Pi_0^2=\Pi_0$, and $\Pi_0\cR=\cR\Pi_0=0$, in particular $\cR(\cB_1)\subset \bV$.

In the first case of the Theorem we have $\sigma(\cR)\subset \{z\in\bC\;:\; |z|<\kappa\}$.
Thus, there exist $A\geq 1$ and $\varrho\in(0,1)$ such that, for all $n\in \bN$, $\|\cR^n\|_{\cB_1}\leq A\varrho^n$.
Let $\bar m\geq \bar n$ to be chosen later. Let $J, H$ be as in Lemma \ref{lem:chain-rule}. For $g\in T\cK$ such that $\|g\|_{\cB_2}\leq \frac 14 J^{-1}H^{-\bar m}$, we have, recalling Lemma \ref{lem:chain-rule},
\[
\begin{split}
\|\widetilde \cL_{\nu}^{\bar m}(h+g)-h\|_{\cB_1}&=\|\widetilde \cL_{\nu}^{\bar m}(h+g)-\widetilde \cL_{\nu}^{\bar m}(h)\|_{\cB_1}\\
&\leq \|\cD_h^{\bar m}(g)\|_{\cB_1}+H^{\bar m}\|g\|_{\cB_2}\|g\|_{\cB_1}\\
&=\|\cR^{\bar m}(g)\|_{\cB_1}+H^{\bar m}\|g\|_{\cB_2}\|g\|_{\cB_1}\\
&\leq \frac 34\|g\|_{\cB_1},
\end{split}
\]
provided $\bar m$ has been chosen large enough.
Moreover, setting $\varrho_*=\max\{\vartheta, \varrho\}$,
\[
\begin{split}
\|\widetilde \cL_{\nu}^{\bar m}(h+g)-h\|_{\cB_2}&\leq J\left( H^{\bar m}\|g\|_{\cB_2} \|g\|_{\cB_1}+ JA\sum_{j=0}^{\bar m-1}\vartheta^{\bar m-j}\varrho^j\|g\|_{\cB_1}+\vartheta^{\bar m} \|g\|_{\cB_2}\right)\\
&\leq \frac 14 \|g\|_{\cB_2}+JA\varrho_*^{\bar m}\bar m\|g\|_{\cB_1}\leq \frac 14 J^{-1}H^{-\bar m}
\end{split}
\]
provided $\bar m$ is large enough. Hence, setting $g_k=\widetilde \cL_{\nu}^{k}(h+g)-h$, we have $\|g_{\bar m}\|_{\cB_1}\leq \frac 34\|g\|_{\cB_1}$ and $\|g_{\bar m}\|_{\cB_2}\leq\frac 14 J^{-1}H^{-\bar m}$, we can thus iterate the estimate.
Accordingly, for all $\|g\|_{\cB_2}\leq \frac 14 J^{-1}H^{-\bar m}$, 
\begin{equation}\label{eq:stconv}
\lim_{n\to\infty}\|\widetilde \cL_{\nu}^n(h+g)-h\|_{\cB_1}=0.
\end{equation}
Since $\cC^\infty\subset \cB_2$, by \eqref{eq:stconv} and  (A\ref{ass:1}), it follows that for an open set of densities in $\cC^\infty$ the sequence of measures converges weakly to $h$. Hence, $h$ is a physical measure. 

On the contrary, if   $\sigma(\cD_h)\setminus\{z\in\bC\;:\; |z|\leq 1\}\neq \emptyset$, by the spectral decomposition, we can further refine the decomposition in \eqref{eq:Rspec} as
\begin{equation}\label{eq:der_split}
\cR(\phi)=A(\phi)+B(\phi),
\end{equation}
$AB=BA=0$, $A$ is finite rank and $\sigma(A)\subset \{z\in\bC\;:\; |z|>\kappa^{-1}\}\cup\{0\}$, $\kappa<1$, while $\sigma(B)\subset \{z\in\bC\;:\; |z|<\kappa\}$.
Let $\bV_A=\operatorname{Range(A)}$, $\bV_B=\operatorname{Range(B)}$, and $\bV_0=\operatorname{Range(\Pi_0)}$, then $\cB_1=\bV_A\oplus\bV_0\oplus\bV_B$. 
Next, notice that equation \eqref{eq:LT_der} implies, for each $\tilde g,\tilde\phi\in\cB_1$,
\begin{equation}\label{eq:V_der_inv}
\int D_{\tilde g}\widetilde \cL_{\nu}(\tilde \phi)=\int \tilde \phi,
\end{equation}
which implies that $\bV=\bV_A\oplus\bV_B$.

Let $\Pi_A\in L(\cB_1,\cB_1)$ the the projection of on $\bV_A$ along $\bV_0\oplus\bV_B$, i.e. $\Pi_A^2=\Pi_A$ and $\Pi_A\cD_h=A \Pi_A$.
Fix $\kappa_1\in (\kappa,1)$ and let $\bar m\in\bN$ be such that $\|A^{\bar m}g\|_{\cB_1}\geq \kappa_1^{-\bar m}\|g\|_{\cB_1}$, for all $g\in\bV_A$,  and $\|B^{\bar m}\|_{\cB_1}\leq \kappa_1^{\bar m}$.

We argue by contradiction. Assume that $h$ is physical. That is, there exists a $\cC^\infty$ open set $\cA$ that converges weakly to $h$. Note that, by eventually restricting $\cA$, we can assume that $\sup_{g\in\cA}\|g\|_{\cC^{\bar r}}<\infty$. 

\begin{lemma}\label{lem:enterBe}
 For each $\nu\in [0,\nu_*]$ and $g\in\cA$, we have
\[
\lim_{n\to\infty}\|\widetilde \cL_{\nu}^ng-h\|_{\cB_2}=0.
\]
Moreover, there exists $\Const>0$ such that, for each $\ve>0$, there exists $m_\ve\in\bN$ such that, for each  $n\geq m_\ve$, $\widetilde \cL_{\nu}^{n}\cA\subset \bB_\ve$,  where 
\[
\bB_\ve:=\{g\in\cB_2\cap\bW\;:\; \|g-h\|_{\cB_2}\leq 2\Const C_2 \ve \ln\ve^{-1}\}.
\]
\end{lemma}
\begin{proof}
By \eqref{eq:embedding} $C_1:=\sup_{g\in\cA}\|g\|_{\cB_3}\leq \sup_{g\in\cA}\|g\|_{\cC^{\bar r}}<\infty$. By Assumption (A\ref{ass:6}), for all $n\in\bN$ and $g\in\cA$,
\[
\|\widetilde \cL_{\nu}^ng\|_{\cB_3}\leq C_*\vartheta^n\|g\|_{\cB_3}+K\|g\|_{\cB_2}\leq (C_*+K)C_1.
\]
Let $C_2=\max\{\|h\|_{\cB_3},(C_*+K)C_1\}$.  
For each $\ve>0$, we defined the weak $\ve$-neighborhood
\[
\cU_\ve=\left\{g\in\cC^0\;:\; \left|\int \alpha (h-g)\right|<\ve\right\}.
\]
Next, by our hypothesis and recalling Definition \ref{def:physical}, for each $\ve>0$ there exists $n_{\ve}\in\bN$ such that 
$\widetilde \cL_{\nu}^n(\cA)\in\cU_{\ve/2}$, for $n\geq n_{\ve}$. Then, for all $m\geq n_\ve$, $n\in\{0,\dots, C_4\ln\ve^{-1}\}$ and $g\in\widetilde \cL_{\nu}^m(\cA)$, let $g_n=\widetilde \cL_{\nu}^{m+n}g$.\\
By equation \eqref{eq:LT_der} and using the Mean value theorem for G\^ateaux differentials \cite[Theorem 3.2.6]{DM13},
\[
\begin{split}
&\|g_{n}-h\|_{\cB_2}=\|\widetilde \cL_{\nu}(g_{n-1})-\widetilde \cL_{\nu}(h)\|_{\cB_2}\leq \int_0^1dt\Big\|\cL_{T_{\nu,h+t(g_{n-1}-h)}}(g_{n-1}-h)\\
&-\nu\divv \cL_{ T_{\nu,h+t(g_{n-1}-h)}}\beta \left[h+t(g_{n-1}-h)\right]\int\alpha(g_{n-1}-h)\Big\|_{\cB_2}\\
&\leq \int_0^1dt \Big\|\cL_{T_{\nu,h+t(g_{n-1}-h)}}(g_{n-1}-h)\Big\|_{\cB_2}+\Const C_2\ve\\
&\leq \int_{[0,1]^n}dt_1\cdots dt_n\Big\|\cL_{T_{\nu,h+t_1(g_{n-1}-h)}}\cL_{T_{\nu,h+t_2(g_{n-2}-h)}}\cdots  \cL_{T_{\nu,h+t_n(g-h)}}(g-h)\Big\|_{\cB_2}\\
&\phantom{=}
+\Const C_2\ve n.
\end{split}
\]
Then, for $\ve$ small enough and $n\geq \bar n$, by Assumption (A\ref{ass:4}) we have
\[
\|g_{n}-h\|_{\cB_2}\leq 2 \vartheta^{n}C_2+\Const C_2  n\ve.
\]
It follows that, choosing $C_4$ large enough, we have
\[
\|g_{C_4\ln\ve^{-1}}-h\|_{\cB_2}\leq 2\Const C_2 \ve \ln\ve^{-1}.
\]
Hence, for all $k\geq n_\ve+C_4\ln\ve^{-1}$ we have $\widetilde \cL_{\nu}^k(\cA)\subset \bB_\ve$.
\end{proof}
Let  $\widetilde \bB_\ve=\bB_\ve\bigcap \widetilde \cL_{\nu}^{m_\ve}\cA$, where $m_\ve$ and $\bB_\ve$ are defined in the statement of Lemma \ref{lem:enterBe}, so that $\widetilde\cL^n(\widetilde \bB_\ve)\subset  \bB_\ve$ for all $n\in\bN$. \\
First of all, we need information on the structure of $\cD_g$ for $g\in \bB_\ve$.
Recall the definition of $\Pi_0$ and $\Pi_A$ just after equation \eqref{eq:V_der_inv}, setting $\Pi_{E}=\Id-\Pi_A$, $\Pi_E\cD_h=:E$, we have 
$\cD_h=A+E$, with $A$ expanding and $E$ power bounded. Also, for each $a>1$, we define the cone

\[
\cC_a=\left\{\xi\in\cB_2\;:\; \|(\Id-\Pi_A)\xi\|_{\cB_1}\leq  \frac 1a\|\Pi_A\xi\|_{\cB_1}\right\},
\]
\begin{lemma}\label{lem:hyper}
There exists $\ve_0, \varrho\in(0,1)$, and $m\in\bN$ such that, for each $\ve\le \ve_0$ and $g\in\widetilde \bB_\ve$, letting $g_k=\widetilde \cL_\nu^kg$, we have $\cD_{g_{m-1}}\cdots \cD_g\cC_2\subset \cC_4$ and, for all $\xi\in\cC_2$,
\[
\|\cD_{g_{m-1}}\cdots \cD_g\xi\|_{\cB_1}\geq \varrho^m\|\xi\|_{\cB_1}.
\]
\end{lemma}
\begin{proof}
Let us start by establishing a Lasota-Yorke type inequality. By Assumption (A\ref{ass:6}) we have, for some constant $C>1$ and $i\in\{1,2\}$
\begin{equation}\label{eq:DLY}
\begin{split}
\|\cD_{g_{m-1}}\cdots \cD_g\xi\|_{\cB_0}&\leq C^m\|\xi\|_{\cB_0}\\
\|\cD_{g_{m-1}}\cdots \cD_g\xi\|_{\cB_i}&\leq \|\cL_{T_{\nu,g_{m-1}}}\cdots \cL_{T_{\nu,g}}\xi\|_{\cB_i}+C^n\|\xi\|_{\cB_{i-1}}\\
&\leq  C_*\vartheta^m\|\xi\|_{\cB_i}+(C^m+K)\|\xi\|_{\cB_{i-1}}.
\end{split}
\end{equation}
If $m$ is such that $C_*\vartheta^m<\frac 12$, this can be iterated yielding the usual Lasota-Yorke inequality. Note that \eqref{eq:DLY} also implies a Lasota-Yorke inequality for $\cD_h$.
In addition, Lemma \ref{lem:regularity} implies that, for $C$ large enough,
\[
\|\cD_{g_{m-1}}\cdots \cD_g\xi-\cD_h^m\xi\|_{\cB_0}\leq C^m\|h-g\|_{\cB_2}\|\xi\|_{\cB_1}.
\]
By \cite[Theorem 1]{KL99}, it follows that the spectrum of $\cD_{g_{m-1}}\cdots \cD_g$ outside the unit disk is close to the one of $\cD_h^m$. Let $\Pi_{A,g}$ be the projector on the part of the spectrum of $\cD_{g_{m-1}}\cdots \cD_g$ outside the unit disk, and set $\Pi_{E,g}=\Id-\Pi_{A,g}$. Then by   \cite[Corollary 1]{KL99} such projectors are close to the eigenprojector $\Pi_A,\Pi_E$ of $\cD_h$. By analytic function calculus, we have, for each $n\in\bN$,
\[
\begin{split}
\|\Pi_A\xi\|_{\cB_1}&=\|\Pi_A^2\xi\|_{\cB_1}=\left\|\frac{1}{2\pi i}\int_\gamma (z-\cD_h)^{-1}\Pi_A\xi dz\right\|_{\cB_1}\\
&=\left\|\frac{1}{2\pi i}\int_\gamma z^{-1}\left[\sum_{k=0}^{n-1}z^{-k}\cD_h^k+z^{-n}(\Id-z^{-1}\cD_h)^{-1}\cD_h^n\right]\Pi_A\xi dz\right\|_{\cB_1}
\end{split}
\]
where $\gamma$ are two circles one of radius $\kappa^{-1}$ and the other of radius $R$ such that $\sigma(\cD_h)\subset\{z\in\bC\;:\; |z|\leq R\}$. Note that  $\int_\gamma z^{-1}\sum_{k=0}^{n-1}z^{-k}\cD_h^kdz=0$. By the Lasota-Yorke \eqref{eq:DLY} it follows that there exists $\Const>0$ such that
\[
\begin{split}
\|\Pi_A\xi\|_{\cB_1}&=\left\|\frac{1}{2\pi i}\int_\gamma z^{-n-1}(\Id-z^{-1}\cD_h)^{-1}\cD_h^n\Pi_A\xi dz\right\|_{\cB_1}\\
&\leq \Const \kappa^{n+1}\left(\vartheta^{n}\|\Pi_A\xi\|_{\cB_1}+(C^n+K)\|\Pi_A\xi\|_{\cB_0}\right).
\end{split}
\]
Hence, choosing $n$ so that $\Const \kappa^{n+1}\vartheta^{n}\leq 1/2$, we obtain that there exists $b>0$ such that 
\[
\|\Pi_A\xi\|_{\cB_0}\geq b\|\Pi_A\xi\|_{\cB_1}.
\]
Then  we have, for some constant $C_3>1$ and $m\geq \bar n$,
\[
\begin{split}
\|\cD_{g_{m-1}}\cdots \cD_g\Pi_A\xi\|_{\cB_1}&\geq \|\cD_{g_{m-1}}\cdots \cD_g\Pi_A\xi\|_{\cB_0}\\
&\geq \|\Pi_A\cD_h^m\Pi_A\xi\|_{\cB_0}-C^mC_3\ve\ln\ve^{-1}\|\Pi_A\xi\|_{\cB_1}\\
&\geq b\|\cD_h^m\Pi_A\xi\|_{\cB_1}-C^mC_3\ve\ln\ve^{-1}\|\Pi_A\xi\|_{\cB_1}\\
&\geq \left[b\kappa_1^{-m}-C^mC_3\ve\ln\ve^{-1}\right]\|\Pi_A\xi\|_{\cB_1}\geq \frac b2\kappa_1^{-m}\|\Pi_A\xi\|_{\cB_1}
\end{split}
\]
provided $\ve$ is chosen small enough. Moreover, by \cite[Corollary 2]{KL99}, for each $r\in (1,\kappa_1^{-1})$ there exist a constant  $C_r>1$ and $n_1\in\bN$, such that, for $m\geq n_1$,
\[
\begin{split}
\|\cD_{g_{m-1}}\cdots \cD_g\Pi_{E}\xi\|_{\cB_1}&\leq \|\cD_{g_{m-1}}\cdots \cD_g\Pi_{E,g}\xi\|_{\cB_1}+\|\cD_{g_{m-1}}\cdots \cD_g(\Pi_E-\Pi_{E,g})\xi\|_{\cB_1}\\
&\leq C_r r^m\|\Pi_{E}\xi\|_{\cB_1}+C_*\vartheta^m\|(\Pi_E-\Pi_{E,g})\xi\|_{\cB_1}\\
&\phantom{\leq}
+(C^m+K)\|(\Pi_E-\Pi_{E,g})\xi\|_{\cB_0}
\end{split}
\]
where in the last line, we have used \eqref{eq:DLY}. Next, \cite[Corollary 2]{KL99} implies that $\|(\Pi_E-\Pi_{E,g})\xi\|_{\cB_0}\leq(C^m+K)^{-1} C_r r^m$, provided $\ve$ is small enough, hence, for $\xi\in\cC_2$,
\[
\begin{split}
\|\cD_{g_{m-1}}\cdots \cD_g\Pi_{E}\xi\|_{\cB_1}&\leq 4 C_r r^m\|\Pi_{E}\xi\|_{\cB_1}\leq 2 C_r r^m\|\Pi_{A}\xi\|_{\cB_1}\\
&\leq \frac {\kappa_1^{m} C_r r^m} b \|\cD_{g_{m-1}}\cdots \cD_g\Pi_A\xi\|_{\cB_1}\\
&\leq \frac 14  \|\cD_{g_{m-1}}\cdots \cD_g\Pi_A\xi\|_{\cB_1}
\end{split}
\]
provided $m$ has been chosen large enough. This provides the invariance for $\cC$. To conclude the lemma, note that for $\xi\in\cC_2$
\[
\begin{split}
\|\cD_{g_{m-1}}\cdots \cD_g\xi\|_{\cB_1}&\geq \|\cD_{g_{m-1}}\cdots \cD_g\Pi_{A}\xi\|_{\cB_1}-\|\cD_{g_{m-1}}\cdots \cD_g\Pi_{E}\xi\|_{\cB_1}\\
&\geq \frac b2\kappa_1^{-m}\|\Pi_A\xi\|_{\cB_1}-4 C_r r^m\|\Pi_E\xi\|_{\cB_1}\\
&\geq  \left(\frac b2\kappa_1^{-m}-2 C_r r^m\right)\|\Pi_A\xi\|_{\cB_1}\geq \frac13\left(b\kappa_1^{-m}-4 C_r r^m\right)\|\xi\|_{\cB_1}.
\end{split}
\]
The above yields the results for $\varrho\in (\kappa_1,1)$, provided $m$ is large enough.
\end{proof}
We can now conclude.
Let $g, g+\phi\in \cL_{\nu}^{m_\ve}\cA$ such that $\phi\in\cC_2$.  Let $g_k=\widetilde\cL_\nu^k(g)$ and $\phi_k=\widetilde\cL_\nu^k(g+\phi)-\widetilde\cL_\nu^k(g)$. Then, Lemma \ref{lem:chain-rule} and Lemma \ref{lem:enterBe}, imply, for some constant $C_3>0$,
\[
\begin{split}
\|\Pi_A\phi_m\|_{\cB_1}\geq& \|\Pi_A\cD_{g_{m-1}}\cdots \cD_{g}\phi\|_{\cB_1}\\
&-\left\|\Pi_A\left[(\widetilde\cL_\nu^m(g+\phi)-\widetilde\cL_\nu(g))-\cD_{g_{m-1}}\cdots \cD_{g}\phi\right]\right\|_{\cB_1}\\
\geq& \|\Pi_A\cD_{g_{m-1}}\cdots \cD_{g}\phi\|_{\cB_1}-H^m\|\Pi_A\|_{\cB_1}\|\phi\|_{\cB_2}\|\phi\|_{\cB_1}.
\end{split}
\]
Note that if $\psi\in\cC_a$, then
\[
(1-a^{-1})\|\Pi_A\xi\|_{\cB_1}\leq \|\xi\|_{\cB_1}\leq (1+a^{-1})\|\Pi_A\xi\|_{\cB_1}.
\] 
Thus Lemma \ref{lem:hyper} implies
\begin{equation}\label{eq:expansion}
\begin{split}
\|\Pi_A\phi_m\|_{\cB_1}&\geq \frac 43\|\cD_{g_{m-1}}\cdots \cD_{g}\phi\|_{\cB_1}-H^m\|\Pi_A\|_{\cB_1}\|\phi\|_{\cB_2}\|\phi\|_{\cB_1}\\
&\geq \varrho^m\frac 45\|\phi\|_{\cB_1}-H^m\|\Pi_A\|_{\cB_1}\|\phi\|_{\cB_2}\|\phi\|_{\cB_1}\\
&\geq  \left(\frac 35 \varrho^m-H^mC_3\ve\ln\ve^{-1}\right)\|\Pi_A\phi\|_{\cB_1}\geq 2\|\Pi_A\phi\|_{\cB_1}
\end{split}
\end{equation}
provided $\ve$ is small enough. Moreover, as already remarked in the proof of Lemma \ref{lem:hyper}, there exists $r\in (1,\rho)$ such that
\[
\begin{split}
\|\Pi_E\phi_m\|_{\cB_1}&=\|\Pi_E(\widetilde\cL_\nu^m(g+\phi)-\widetilde\cL_\nu(g))\|_{\cB_1}\\
& \leq \|\Pi_E\cD_{g_{m-1}}\cdots \cD_{g}\phi\|_{\cB_1}+H^m\|\Pi_E\|_{\cB_1}\|\phi\|_{\cB_2}\|\phi\|_{\cB_1}\\
&\leq \frac 14\|\Pi_A\cD_{g_{m-1}}\cdots \cD_{g}\phi\|_{\cB_1}+H^m\|\Pi_E\|_{\cB_1}\|\phi\|_{\cB_2}\|\phi\|_{\cB_1}\\
&\leq  \frac 14\|\Pi_A\phi_m\|_{\cB_1}+2H^mC_3\ve\ln\ve^{-1}\|\Pi_A\phi\|_{\cB_1}\\
&\leq  \frac 14\|\Pi_A\phi_m\|_{\cB_1}+H^mC_3\ve\ln\ve^{-1}\|\Pi_A\phi_m\|_{\cB_1}\\
&\leq \frac 12\|\Pi_A\phi_m\|_{\cB_1}
\end{split}
\]
provided $\ve$ is small enough and where, in the line next to the last, we have used \eqref{eq:expansion}. Hence $\phi_m\in\C_2$, and the argument can be iterated.Thus, by \eqref{eq:expansion}
\[
\|\phi_{km}\|_{\cB_1}\geq \frac 12 \|\Pi_A \phi_{km}\|_{\cB_1}\geq 2^{k-1}\|\Pi_A\phi\|_{\cB_1}\geq \frac{2^k}3\|\phi\|_{\cB_1},
\]
which, for $k$ large, contradicts $g_{km}+\phi_{km}\in\bB_\ve$.

Therefore, $h$ is not a physical measure.
\end{proof}

\section{Example}\label{sec:example}
In this section, we present a nontrivial example to which our theory applies. More precisely, we provide an example of an infinite dimensional system where for very small coupling strength, the system admits a unique physical measure, while for not so small coupling strength, the system admits at least three physical measures, see Figure \ref{NewBirfucation} for an illustration. We also discuss which measures are physical. On the contrary, when considering the finite-dimensional version of our example with the same coupling strength, we show that the system admits a unique physical measure. Hence, the existence of multiple physical states is a purely infinite dimensional phenomenon, i.e., the systems exhibit a phase transition.

Let $A\in\operatorname{SL}(2,\bN)$ be a hyperbolic matrix with positive integer entries. Let  $T_0(x)=Ax\mod 1$ the corresponding linear automorphism with eigenvalues $\lambda,\lambda^{-1}$, $\lambda>1$.
Let $\rho\in \operatorname{Diff}^\infty(\bT^2)$, and $n_*\in\bN$, then $T=\rho \circ T_0^{n_*}\circ \rho^{-1}$ is an Anosov diffeormopfism of $\bT^2$.
If $J=\det(D\rho)$, then it can be checked by an explicit computation that $h_*=\left(J\circ \rho^{-1}\right)^{-1}$ is the SRB measure of $T$.
Let $\beta\in\bR^2$, $\beta=(\sin\vartheta,\cos\vartheta)$, $\vartheta\in(0,\pi)\cup(\pi,2\pi)$, and $\frac{\beta_1}{\beta_2}\not\in\bQ$.

For each $h\in\cC^\infty(\bT^2,\bR_+)$, $\int h=1$, we define
\footnote{ These are just special choices that allow us to exhibit an example for which we easily can perform explicit computation; many other choices are as good.} 
\begin{equation}\label{eq:exmap}
\begin{split}
&\alpha(x)=1-\cos\langle \beta,x\rangle\\
&\Phi_{\nu,h}(x)=x+\nu\beta\int\alpha h\\
&T_{\nu,h}(x)=\Phi_{\nu,h}\circ T(x) \quad\text{mod }1.
\end{split}
\end{equation}
Also, we assume
\begin{equation}\label{eq:rho}
\rho(x)=\rho_\mu(x)=x+\mu\chi(x)\mod 1,
\end{equation}
where $\chi(x)=(\cos \langle \beta,x\rangle, \sin \langle \beta,x\rangle)$, $\mu\in\bR$, to be chosen later.
\begin{figure}[ht]
   \centering
   \includegraphics [scale=0.44]{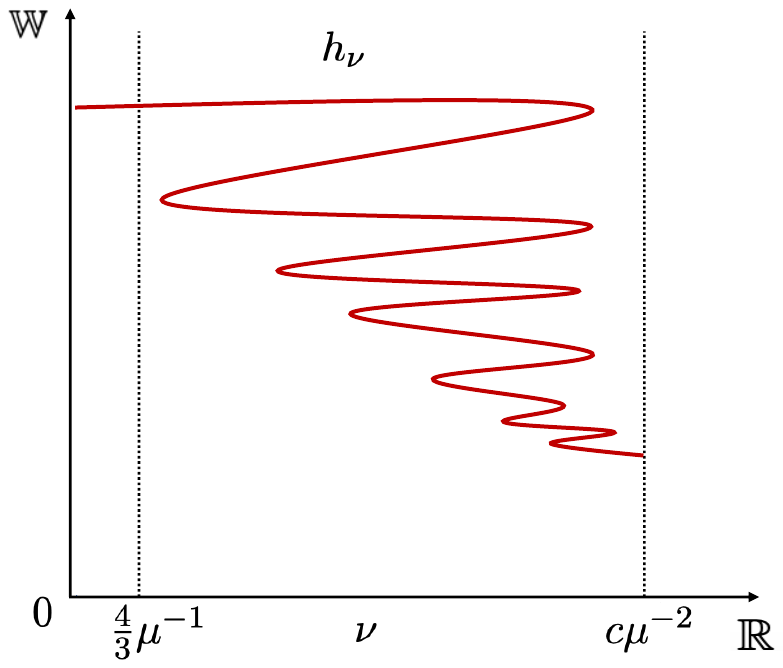} 
   \vskip -2cm
    \caption{Diagram corresponding to Proposition \ref{prop:exsols}.\\
     The red curve is the set of the invariant measures $h_\nu$.}   
\end{figure} 
The main result of this section is summarized in the following Proposition, whose proof is the content of Section \ref{sec:ex_bifurc}.
\begin{proposition}\label{prop:exsols}
There exist $\mu_\star, c_+,c>0$ such that for each $\mu\leq\mu_\star$, and $n_*\geq c_+\ln\mu^{-1}$ we have:
for $\nu\leq \frac 43 \mu^{-1}$ small enough the system defined in \eqref{eq:exmap} admits a unique physical measure. While, for each $\nu\in (2\mu^{-1},c\mu^{-2})$, but for finitely many values, the system admits at least $2m+1$, $m\geq 1$, invariant measures, only $m+1$ of which are physical measures.
\end{proposition}
The content of Proposition \ref{prop:exsols} is illustrated pictorially in Figure 3 where it is also shown that the curve on which the invariant measures lie has a lot of turning points. Indeed, by Lemma \ref{lem:singular}, the number of turning points is proportional to $\mu^{-2}$.\\
A natural question is if the existence of many physical measures is already present when considering a finite system or if it is a genuine property of the infinite system; what in physics is called a phase transition. In Section \ref{sec:phase}, we prove that we are indeed in the presence of a phase transition.

\subsection{The bifurcation diagram}\label{sec:ex_bifurc}\ \\
This section is devoted to the proof of Proposition \ref{prop:exsols}.
We start with the properties of the maps $T_{\nu,h}$; then we introduce the appropriate Banach spaces and prove that the general setting applies to the present case.
\begin{lemma}\label{lem:diffdist}
For each $n_*\in\bN$, there exists $\mu_*>0$ such that for each $\nu\geq 0$, $\mu\in [0,\mu_*)$, and $h\in\cC^\infty(\bT^2,\bR_+)$, $T_{\nu,h}$ is an  Anosov diffeomorphism with uniform\footnote{i.e., independent of $\nu$.} hyperbolicity.
 \end{lemma}
\begin{proof}
By definition if $T_{\nu,h}(x)=T_{\nu,h}(y)$, then $T(x)=T(y)$ and hence $x=y$, so $T_{\nu,h}$ is a diffeomorphism. On the other hand 
\[
DT_{\nu,h}(x)=DT=D\rho_\mu\circ T^{n_*}_0\circ\rho_\mu^{-1}\cdot D T^{n_*}_0\circ\rho_\mu^{-1} \cdot D\rho^{-1}.
\]
Since $A$ is hyperbolic, it has an expanding direction and a contracting direction, let us call the $v^u,v^s$, the corresponding eigenvectors, respectively. We can assume without loss of generality that $\|v^u\|=\|v^s\|=1$. Given the cones $\cC_i=\{av^u+bv^s\in \bR^2\;:\; |b|\leq i |a|\}$, we have $ T^{n_*}_0\cC_2\subset \cC_1$ and, for all $v\in\cC_2$ $\|DT^{n_*} v\|\geq 2\|v\|$, provided $n_*$ is large enough. Let $\cC_*$ be a cone such that $\cC_2\subset \cC_*\subset\cC_1$,  where the inclusions are strict. We can then choose $\mu_*$ such that, for all $\mu\in [0,\mu_*)$, $D\rho_\mu\cC_1\subset \cC_*$ and $D\rho_\mu^{-1}\cC_*\subset \cC_2$. Accordingly, $DT_{\nu,h}\cC_*\subset \cC_*$, strictly. In addition, there exists $c>0$ such that, for $v\in \cC_*$
\[
\|DT_{\nu,h}v\|\geq 2(1-c|\mu|)\|v\|
\]
hence, the vectors in the cone expand provided $\mu$ is small enough. The hyperbolicity claim follows.
\end{proof}
Let $\cB_i=\cB^{i,q-i}$, $i=0,\dots,3$, $q-i\in\bR_+$, $q\leq\bar r-2$, be the Banach spaces defined in \cite[Sections 2 and 3]{GL06} using the cone $\cC_*$, defined in the proof of Lemma \ref{lem:diffdist}, instead of the cone defined at the beginning of  \cite[Sections 3]{GL06}.\footnote{Other Banach spaces, e.g. \cite{BL22}, would work as well. We use  \cite{GL06} for convenience. On the contrary, we could not use the spaces in \cite{BKL} since it does not allow us to treat operators describing different dynamics acting on the same space.} We will also use the Banach spaces $\cB_i(\cC_k)=\cB^{i,q-i}(\cC_k)$ for the same construction using the cones $\cC_k$. Note that, $\cB_i(\cC_2)\supset \cB_i\supset \cB_i(\cC_1)$.

\begin{lemma}\label{lem:uniLYex}
 There exists $C>0$  and $\vartheta\in(0,1)$ such that, for all $g\in\cB_i$, $i\in\{1,\dots,3\}$,
\[
\|\cL_{T_{\nu,h}}  g\|_{\cB_i}\leq \vartheta\|g\|_{\cB_i}+C\|g\|_{\cB_{i-1}}.
\]
\end{lemma}
\begin{proof}
Note that
\[
\cL_{T_{\nu,h}}=\cL_{\Phi_{\nu,h}}\cL_{\rho}\cL_{T_0^{n_*}}\cL_{\rho^{-1}}.
\]
 By the definition of the Banach spaces it follows that, for $\mu$ small enough and $n_*$ large, we have
\begin{equation}\label{eq:intineq}
\begin{split}
&\|\cL_{\Phi_{\nu,h}}\|_{\cB_{i}}\leq C\\
&\|\cL_{T^{n_*}}\|_{\cB_{i}(\cC_2)\to\cB_{i}(\cC_1)}\leq 1\\
&\|\cL_{\rho}\|_{\cB_{i}(\cC_1)\to\cB_{i}}\leq C\\
&\|\cL_{\rho^{-1}}\|_{\cB_{i}\to\cB_{i}(\cC_2)}\leq C.\\
\end{split}
\end{equation}

Moreover, by Lemma 2.2 of \cite{GL06} we have
\begin{equation}\label{eq:basineq}
\|\cL_{T_0^{n_*}}g\|_{\cB_{i}(\cC_1)}\leq C\lambda^{-n_*\min\{i,q-i\}}\|g\|_{\cB_{i}(\cC_2)}+C\|g\|_{\cB_{i-1}(\cC_2)}.
\end{equation}
The lemma follows by \eqref{eq:intineq} and \eqref{eq:basineq}.
\end{proof}

We verify assumptions (A\ref{ass:0})-(A\ref{ass:4}) for the system defined in \eqref{eq:exmap} using these Banach spaces. Indeed, (A\ref{ass:0}) holds by definition of $\cB_{i}$.  For assumption (A\ref{ass:1}),  the fact that for any $a\in \cC^{\bar r}$, $\|a h\|_{\cB_{i}}\leq C_* \|a\|_{\cC^{q}}\|h\|_{\cB_{i}}$ follows as in \cite[Lemma 3.2]{GL06}, while the fact that $\left|\int h\right|\leq C_* \|h\|_{\cB_{0}}$ follows by definition of $\cB_{0}$ (see \cite[Section 4 ]{GL06}). Also, (A\ref{ass:3}) is verified in  \cite[Lemma 2.1]{GL06}. (A\ref{ass:2}) and 
(A\ref{ass:6}) follow from Lemma \ref{lem:uniLYex} above. Notice $\cL_{T_{\nu,h}}$ admits a spectral gap on  $\cB_{i}$ for $i\in\{1,2,3\}$. This follows by \cite[Theorem 2.3]{GL06}, since $T_{\nu,h}$ is a toral Anosov map and hence topologically transitive, see \cite[Proposition 18.6.5]{KH95}. Consequently,  $\cL_{T_{\nu,h}}$ admits a spectral decomposition; i.e., for any $g\in \cB_{i}$, $i\in\{1,2,3\}$, 
\begin{equation}\label{eq:specex}
\cL_{T_{\nu,h}}g= h\int g+Q_{T_{\nu,h}}g.
\end{equation}
We delay the verification of assumption (A\ref{ass:4}) till later, namely after the proof of Lemma \ref{lem:singular} below. 
Before stating and proving the next lemma, we first note that, \eqref{eq:rho} implies 
\begin{equation}\label{eq:rhoinv}
\rho_\mu^{-1}(x)=x-\mu\chi(x)+\cO(\mu^2).
\end{equation}
Thus,
\begin{equation}\label{eq:J}
\begin{split}
&J \circ \rho_\mu^{-1}=e^{-\operatorname{Tr}\ln(\Id+\mu D\chi)\circ\rho_\mu^{-1}}=e^{-\operatorname{Tr}(\mu D\chi-\frac {\mu^2}2 (D\chi)^2+\cO(\mu^3))\circ\rho_\mu^{-1}}\\
 &=\left[1-\mu\operatorname{Tr}(D\chi)+\frac{\mu^2}2\left[\operatorname{Tr}((D\chi)^2)+[\operatorname{Tr}(D\chi)]^2 \right]+\cO(\mu^3)\right]\circ\rho_\mu^{-1}\\
 &=1-\mu\operatorname{Tr}(D\chi)+\left[\operatorname{Tr}((D\chi)^2)+[\operatorname{Tr}(D\chi)]^2+2\langle \nabla \operatorname{Tr}(D\chi), \chi\rangle\right]\frac{\mu^2}2\\
 &\phantom{=}
 +\cO(\mu^3)=1+a\mu+b\mu^2+\cO(\mu^3)
\end{split}
 \end{equation}
 where 
 \begin{equation}\label{eq:ab}
 \begin{split}
a&=-\operatorname{Tr}(D\chi)=\beta_1\sin\langle \beta,x\rangle-\beta_2\cos \langle \beta,x\rangle\\
b&=\operatorname{Tr}((D\chi)^2)+[\operatorname{Tr}(D\chi)]^2+2\langle \nabla \operatorname{Tr}(D\chi), \chi\rangle\\
& =(\beta_1^2-\beta_2^2)\sin^2\langle \beta,x\rangle-(\beta_1^2-\beta_2^2)\cos^2\langle \beta,x\rangle-\frac 32\beta_1\beta_2\sin(2 \langle \beta,x\rangle).
\end{split}
\end{equation}
\begin{lemma}\label{lem:notzero}
There exists $\mu_\star>0$ such that for all $\mu\leq \mu_\star$, 
$\Theta(\nu, h)\not\equiv 0$.
\end{lemma}
\begin{proof}
Recalling \eqref{eq:Theta}, we have 
\begin{equation}\label{eq:Theta_ex}
\Theta(\nu,h)=-(\Id-\cL_{T_{\nu,h}})^{-1}\partial_\beta H(\nu,h).
\end{equation}
Thus, if  $\Theta(\nu, h)\equiv 0$, then, seeing $H(\nu,h)$ as a distribution, we have
\[
0=\int e^{i\langle k,x\rangle}\partial_\beta H(\nu,h)=-i\langle k,\beta\rangle\int e^{i\langle k,x\rangle} H(\nu,h).
\] 
Since $\beta$ is an irrational vector $\langle k,\beta\rangle\neq 0$ for all $k\in\bZ^2\setminus\{0\}$. It follows that $H(\nu,h)-1$ is zero against the dense set $\{e^{i\langle k,x\rangle}\}_{k\in\bZ^2}$, hence $H(\nu,h)=1$ as a distribution, and hence as an element of $\cB_1$. However, for each $\vf\in\cC^0$, recalling that the SRB measure of $T$ is $(J\circ\rho^{-1})^{-1}\Leb$, we can write, setting $\eta_h:=\beta\int\alpha h$
\[
\begin{split}
\int\vf&=\int \vf\circ T_{\nu,h}=\int\vf\left(T(x)+\nu\beta\int\alpha h\right)\frac{J\circ\rho^{-1}(x)}{J\circ\rho^{-1}(x)}\\
&=\int\vf\left(x+\nu\eta_h\right)\frac{J\circ\rho^{-1}\circ T^{-1}(x)}{J\circ\rho^{-1}(x)}
=\int\vf\left(x\right)\frac{J\circ\rho^{-1}\circ T^{-1}(x-\nu\eta_h)}{J\circ\rho^{-1}(x-\nu\eta_h)}.
\end{split}
\]
It follows that $J\circ \rho^{-1}$ must be an invariant function for $T$. But, $T$ being ergodic, $J\circ \rho^{-1}$ must then be constant  contradicting equations \eqref{eq:J} and \eqref{eq:ab}.
 \end{proof}
\begin{lemma}\label{lem:singular}
There exist $\mu_\star, c_+,c>0$ such that for each $\mu\leq\mu_\star$, and $n_*\geq c_+\ln\mu^{-1}$, any continuous curve $(\nu(\tau),h_{\nu(\tau)})$ in $\bR\times(\cB_2\cap \bW)$ such that $\cL_{T_{\nu,h_\nu}}h_\nu=h_\nu$, the equation
\begin{equation}\label{eq:zeroex}
\nu\int\alpha\Theta(\nu, h_\nu)=1
\end{equation}
has no solutions for $\nu\leq \frac 4{3\mu}$ and at least two solutions in each interval $I$ such that $\nu(I)\subset (\frac 5{3\mu}, \frac c{\mu^{2}})$ and $\nu(I)$ has length larger than $2\pi+c_+$.
\end{lemma}
\begin{proof}
For each $\mu>0$, choosing $c_+$ large enough, we have
\begin{equation}\label{eq:chosen}
\left\|\cL_{T_0^{n_*-1}} g-\int g\right\|_{\cB_3}\leq \mu^3\|g\|_{\cB_3}.
\end{equation}
This implies that, for all smooth $\vf$ and $g\in\bW$,
\[
\begin{split}
\int \vf\cL_{T_{\nu,h}}g&=\int g(x)\vf\circ \Phi_{\nu,h}\circ\rho_\mu\circ T_{0}^{n_*}\circ \rho_\mu^{-1}(x)dx\\
&=\int g\circ \rho_\mu(y)J(y)\vf\circ \Phi_{\nu,h}\circ\rho_\mu\circ T_{0}^{n_*}(y)dy\\
&=\int \left[\cL_{T_0^{n_*-1}}[g\circ \rho_\mu J] -\int g\circ \rho_\mu J\right]\vf\circ \Phi_{\nu,h}\circ\rho_\mu\circ T_0\\
&\phantom{=}
+\int g\int \vf\circ \Phi_{\nu,h}\circ\rho_\mu.
\end{split}
\]
Since 
\[
\int \vf\circ \Phi_{\nu,h}\circ\rho_\mu=\int \frac{\vf}{J \circ \rho_\mu^{-1}\circ\Phi_{-\nu,h}},
\]
we have
\[
\cL_{T_{\nu,h}}g=\frac{\cL_{T_0}\left[\cL_{T_0^{n_*-1}}[g\circ \rho_\mu J] -\int g\circ \rho_\mu J\right]}{J \circ \rho_\mu^{-1}\circ\Phi_{-\nu,h}}
+\frac{1}{J \circ \rho_\mu^{-1}\circ\Phi_{-\nu,h}}\int g.
\]
The first term after the equality can be estimated using \eqref{eq:chosen}, yielding, for $i\in\{1,2,3\}$,
\begin{equation}\label{eq:Laction}
\left\|\cL_{T_{\nu,h_\nu}}g-\frac{1}{J \circ \rho_\mu^{-1}\circ\Phi_{-\nu,h}}\int g\right\|_{\cB_i}\leq C\mu^3\|g\|_{\cB_i}.
\end{equation}
Note that, for $\mu$ small $g\circ \rho_\mu\in\cB_1(\cC_2)$ and so the estimate follows remembering also \eqref{eq:intineq}. Thus, recalling \eqref{eq:J},\footnote{ Recall that $h_\nu$ may not be a function but only a distribution, hence $\cO$, is meant with respect to the norm of $\cB_2$.}
\begin{equation}\label{eq:h}
\begin{split}
&h_\nu=\cL_{T_{\nu,h_\nu}}h_\nu=\frac{1}{J \circ \rho_\mu^{-1}\circ\Phi_{-\nu,h_\nu}}+\cO(\mu^3\|h_\nu\|_{\cB_2})\\
&\phantom{h_\nu}=1-\mu a\circ \Phi_{-\nu,h_\nu}+\mu^2\left[ \frac{a^2}2-b\right]\circ \Phi_{-\nu,h_\nu}+\cO(\mu^3\|h_\nu\|_{\cB_2})\\
&\phantom{h_\nu} =1-\mu a\circ \Phi_{-\nu,h_\nu}+\mu^2\left[ \frac{a^2}2-b\right]\circ \Phi_{-\nu,h_\nu}+\cO(\mu^3).
\end{split}
\end{equation}
Moreover, by \eqref{eq:Laction}, for $g\in\bV$, and $i\in\{1,2,3\}$,
\begin{equation}\label{eq:Q}
\begin{split}
&\left\|Q_\nu g\right\|_{\cB_i}=\left\|\cL_{T_{\nu,h_\nu}}g-h_\nu\int g\right\|_{\cB_i}\\
&\phantom{\left\|Q_\nu g\right\|_{\cB_2}}
=\left\|\cL_{T_{\nu,h_\nu}}g-\frac{1}{J \circ \rho_\mu^{-1}\circ \Phi_{-\nu,h_\nu}}\int g\right\|_{\cB_1}\leq C\mu^3\|g\|_{\cB_i}.
\end{split}
\end{equation}
We assume that $\mu_\star$ is so small that $C\mu_\star^3\leq \frac 12$.
Thus, setting $\omega=\int \alpha h_\nu$, and using Werner formulae,
\begin{equation}\label{eq:iterone}
\begin{split}
&\omega=\int_{\bT^2} \alpha(x) h_\nu(x) dx=\int \alpha \left\{1- \mu a\circ \Phi_{-\nu,h}\right\}+\cO(\mu^2)\\
&=1+\mu\int \cos\langle\beta,x+\nu\beta\omega\rangle \left(\beta_1\sin\langle \beta, x\rangle-\beta_2\cos\langle \beta, x\rangle\right)+\cO(\mu^2)\\
&=1+\frac\mu 2\int \beta_1\left[\sin\langle\beta,2x+\nu\beta\omega\rangle- \sin\langle\beta,\nu\beta\omega\rangle\right]\\
&\phantom{=}
-\frac\mu 2\int \beta_2\left[\cos\langle\beta,2x+\nu\beta\omega\rangle+ \cos\nu \omega\right]
+\cO(\mu^2)\\
&=1-\frac{\mu}2\left(\sin\vartheta\sin\nu \omega+\cos\vartheta\cos\nu \omega\right)+\cO(\mu^2)\\
&=1-\frac{\mu}2\cos(\nu\omega-\vartheta)+\cO(\mu^2).
\end{split}
\end{equation}
This implies $\left|\omega-1\right|\leq C|\mu|$.
Recalling \eqref{eq:Theta_ex}, \eqref{eq:Q} and \eqref{eq:h},
\begin{equation}\label{eq:turning}
\begin{split}
&\int\alpha\Theta(\nu, h_\nu)=-\int\alpha(\Id-Q_{\nu})^{-1}\partial_\beta h_\nu\\
&=-\int\alpha\partial_\beta h_\nu +\cO(\mu^2)=\int\partial_\beta\alpha\cdot h_\nu+\cO(\mu^2)\\
&=-\mu\int \sin \langle \beta,x+\nu\beta\omega\rangle \left[  \beta_1\sin\langle \beta,x\rangle-\beta_2\cos \langle \beta,x\rangle\right]+\cO(\mu^2)\\
&=-\frac{\mu}2\left[\beta_1\cos\nu \omega-\beta_2\sin\nu\omega\right]+\cO(\mu^2)\\
&=\frac{\mu}2\sin(\nu\omega-\vartheta)+\cO(\mu^2)
\end{split}
\end{equation}
We can now conclude our study of equation \eqref{eq:zeroex}.

\begin{equation}\label{eq:nosol}
\begin{split}
\Gamma(\nu):=&\nu\int\alpha\Theta(\nu, h_\nu)\\
=&\frac{\mu\nu}2\sin\left(\nu-\vartheta-\nu\frac{\mu}2\cos(\nu\omega-\vartheta)+\cO(\nu\mu^2)\right)+\cO(\nu\mu^2)\\
=&\frac{\mu\nu}2\sin\left(\nu(1+\cO(\mu))-\vartheta\right)+\cO(\nu\mu^2).
\end{split}
\end{equation}
Thus, for $\mu$ small enough and $\nu\leq \frac 4{3\mu}$ we have $|\Gamma(\nu)|<1$. On the other hand, if $\nu\in (\frac 5{3\mu}, \frac {c}{\mu^{2}})$, for some constant $c>0$ small enough, we have that the equation \eqref{eq:zeroex} must have at least two solutions each time that $\nu$ varies more than $2\pi+c_+$.

\end{proof}
We now verify (A\ref{ass:4}). Recalling the spectral decomposition of $\cL_{T_{\nu,h}}$ in \eqref{eq:specex}, which holds for $\cB_i$, with $i\in\{1,2,3\}$, inequality \eqref{eq:Q} yields for $g\in\cB_i$ and $\mu$ small, 
$$\|Q_{T_{\nu,h}}\|_{\cB_{i}}\leq \sigma$$
uniformly, for some $\sigma\in(0,1)$. Thus, if $\int g=0$, $\|g\|_{\cB_{i}}\leq 1$, then for all $\{S_j\}_{j=1}^{\bar n}\subset \cT$ 
\[
\|\cL_{S_{\bar n}}\cdots\cL_{S_1}g\|_{\cB_{i}}=\|Q_{S_{\bar n}}\dots Q_{S_1} g\|_{\cB_{i}}\leq \sigma^{\bar n},
\]
which verifies (A\ref{ass:4}) and shows that the present model satisfies all our assumptions.
\begin{lemma}\label{lem:vercond2}
If $h_\nu \in\bW$ is such that $\cL_{T_{\nu,h_\nu}}h_\nu=h_\nu$,
and
\[
\nu\int\alpha\Theta(\nu, h_\nu)=1,
\]
then, for $\nu \in (4\mu^{-1}, c\mu^{-2})$, with $c$ small enough,
\[
\begin{split}
& \operatorname{Range}(D_{h_{\nu}} F)\oplus D_{\nu} F=\bV\\
& \int\alpha D^2_hF|_{\nu,h_{\nu}}(\Theta(\nu,h_{\nu}),\Theta(\nu,h_{\nu}))=\cO(\mu)\neq 0\\
&\int \alpha h_{\nu}\neq 0.
 \end{split}
\]
\end{lemma}

\begin{proof}
For any $g\in\bV$, let
\begin{equation}\label{eq:range1}
\phi=\nu\Theta(\nu, h_\nu)\cdot a\int\alpha h_{\nu} -g,
\end{equation}
with $$a=\frac{\int\alpha g}{\int \alpha h_{\nu}}.$$ 
Notice that, by definition, $\alpha> 0$ outside of a Lebesgue null set, hence, for $\mu$ small enough, $\int \alpha h_{\nu}>0$, hence $a$ is well defined. This proves the third statement of the Lemma. Note that, by the hypothesis and the definition of $\phi$ it follows $\int\alpha\phi=0$. Recalling Lemma \ref{lem:der} we have
\[ 
\begin{split}
&D_{h} F|_{(\nu, h_\nu)}\phi=\phi\\
&D_{\nu} F|_{(\nu, h_\nu)}=-\Theta(\nu, h_\nu)\int\alpha h_{\nu}.
\end{split}
\]
Consequently \eqref{eq:range1} can be written as
\begin{equation}\label{eq:range2}
g=-\nu aD_{\nu} F|_{(\nu, h_\nu)}-D_{h} F|_{(\nu, h_\nu)}\phi,
\end{equation}
which proves the first statement of the Lemma.

To prove the second statement we use Lemma \ref{lem:der}. Since $\cL_{T_{\nu,h_\nu}}h_\nu=h_\nu$, $\nu\int\alpha\Theta(\nu, h_\nu)=1$, we get
\begin{equation}\label{eq:checksec}
\begin{split}
&D^2_hF|_{\nu,h_{\nu}}(\Theta(\nu,h_{\nu}),\Theta(\nu,h_{\nu}))\\
&=-\Big\{2\sum_{i,j}\beta_i\beta_j\left(\Id-\cL_{T_{\nu,h}}\right)^{-1} \partial_{x_i}\left[ \cL_{ T_{\nu,h_\nu}}(\Id-\cL_{T_{\nu,h}})^{-1}\partial_{x_j}h_\nu\right]\\
&+(\Id-\cL_{T_{\nu,h}})^{-1}\sum_{i,j}\beta_i\beta_j\partial_{x_i}\partial_{x_j} h_\nu \Big\}.\\
&=2\left(\Id-\cL_{T_{\nu,h}}\right)^{-1} \partial_{\beta} \cL_{ T_{\nu,h_\nu}}\Theta(\nu,h_{\nu})+\left(\Id-\cL_{T_{\nu,h}}\right)^{-1}\partial_\beta^2h_\nu
\end{split}
\end{equation}
Therefore, by \eqref{eq:checksec} and \eqref{eq:Q}, \eqref{eq:Laction}, \eqref{eq:h}, and \eqref{eq:iterone} we obtain
\[
\begin{split}
&\int\alpha D^2_hF|_{\nu,h_{\nu}}(\Theta(\nu,h_{\nu}),\Theta(\nu,h_{\nu}))\\
&=\int\alpha \left[2\left(\Id-Q_{T_{\nu,h}}\right)^{-1} \partial_{\beta} \cL_{ T_{\nu,h_\nu}}\Theta(\nu,h_{\nu})+\left(\Id-Q_{T_{\nu,h}}\right)^{-1}\partial_\beta^2h_\nu\right]\\
&=\int\alpha \left[2 \partial_{\beta} \cL_{ T_{\nu,h_\nu}}\Theta(\nu,h_{\nu})+\partial_\beta^2h_\nu\right]+\cO(\mu^3)\\
&=-2\int(\partial_\beta\alpha) Q_{ T_{\nu,h_\nu}}\partial_\beta h_\nu+\int(\partial_\beta^2\alpha) h_\nu+\cO(\mu^3)\\
&=-\mu\int\cos(\langle\beta,x\rangle \langle[\beta_1\sin\langle\beta,x\rangle-\nu\omega)-\beta_2\cos(\langle\beta, x\rangle-\nu\omega)]+\cO(\mu^3)\\
&=\frac{\mu}2 \beta_1\sin\nu\omega+\frac{\mu}2 \beta_2\cos\nu\omega+\cO(\mu^3)\\
&=\frac{\mu}2 \cos(\nu\omega-\vartheta)+\cO(\mu^3).
\end{split}
\]
By equation \eqref{eq:turning} we have that $\nu\int\alpha\Theta(\nu, h_\nu)=1$ when 
\[
|\sin(\nu\omega-\vartheta)|=\frac{2}{\nu\mu}+\cO(\mu)\leq \frac34.
\]
Hence, the Lemma.
\end{proof}
\begin{remark}\label{re:turn}
Lemma \ref{lem:vercond2} implies that when $\nu(\bar\tau)\int\alpha\Theta(\nu(\bar \tau), h_{\nu(\bar \tau)}=1$, then the curve $(\nu(\tau), h_{\nu(\tau)})$ has a turning point at $\bar\tau$ in the bifurcation diagram of the invariant measures, see figure \ref{NewBirfucation}.
\end{remark}
We are now ready to conclude the argument.
\begin{proof}[{\bf Proof of Proposition \ref{prop:exsols}}]\ \\
Lemma \ref{lem:noturn}, Lemma \ref{lem:singular}, and Lemma \ref{lem:vercond2} show that the system defined in  \eqref{eq:exmap} satisfies the assumptions of Theorem \ref{thm:implicitg}. 
We can thus consider the curve starting from the unique physical invariant measure at $\nu=0$ and follow it till its first turning point. Note that, by Lemma \ref{lem:noturn}, there can be only one invariant measure, and hence no turning points, for all $\nu\leq \frac 4{3\mu}$. Hence, once we reach a turning point, the curve cannot go back to $\frac 4{3\mu}$ and must have another turning point. Let $\nu_1$ be the first turning by Lemma \ref{lem:singular} the curve must meet another turning point $\nu_2\geq \nu_1-2\pi-c_+$. We can then keep following the curve till the next turning point $\nu_3$ that, always by Lemma \ref{lem:singular} must occur before $\nu_2+2\pi+c_+$. Note that Theorem \ref{thm:implicitg} shows that the curve cannot backtrack in the interval $(\nu_2,\nu_1)$ infinitely many times, hence the curve must keep moving toward increasing values of $\nu$ till it exists the interval $(0,c\mu^{-2})$. The curve will have at least $C\mu^{-1}$ turning points $\nu_i$, for some constant $C>0$.
Lemma \ref{lem:singular} implies then that for each $\nu\geq \frac 53\mu^{-1}$, $\nu\neq \nu_i$, there must at least three invariant measures.

To verify which of the invariant measures are physical, we use Proposition \ref{prop:stability} and Theorem \ref{lem:physical}. Recall that equation \eqref{eq:Q} implies that, for $\mu$ small enough, the essential spectral radius of $\cL_{T_{\nu, h_\nu}}$ is smaller than $1/2$.
We can compute as in \eqref{eq:turning} and obtain, for $z\in \bC$, $|z|> 1/2$,
\begin{equation}\label{eq:physical}
\begin{split}
&\Xi(z)=\int\alpha\Theta(z)=-\nu z^{-1}\int\alpha(\Id -z^{-1}\cL_{T_{\nu, h_\nu}})^{-1}\partial_\beta h_\nu.
\end{split}
\end{equation}
This allows us to compute
\begin{equation}\label{eq:physical1}
\begin{split}
&\Xi(z)=\int\alpha\Theta(z)=-\nu z^{-1}\int\alpha(\Id -z^{-1}Q_{\nu})^{-1}\partial_\beta h_\nu\\
&=-\nu z^{-1}\int\alpha\partial_\beta h_\nu +\cO(\nu\mu^3)=\nu z^{-1}\int\partial_\beta\alpha\cdot h_\nu+\cO(\mu^2)\\
&=-\nu z^{-1}\mu\int \sin \langle \beta,x+\nu\beta\omega\rangle \left[  \beta_1\sin\langle \beta,x\rangle-\beta_2\cos \langle \beta,x\rangle\right]+\cO(\nu\mu^3)\\
&=-\nu \frac{\mu}{2 z}\left[\beta_1\cos\nu \omega-\beta_2\sin\nu\omega\right]+\cO(\nu\mu^3)\\
&=\nu \frac{\mu}{2 z}\sin(\nu\omega-\vartheta)+\cO(\nu\mu^3).
\end{split}
\end{equation} 
We are interested in computing the maximal eigenvalue of $\cD_h$ near a turning point $\tau_0$, where $\Xi(1)=1$. For $\tau\neq \tau_0$, Proposition \ref{prop:stability} implies that 1 is an eigenvalue and the other eigenvalues of modulus bigger than 1/2 are the solutions of the equation $\Xi(z)=1$.  Note that \eqref{eq:physical1} implies that, for $|\nu-\nu_0|\leq c_1\mu$, $c_1$ small enough, $\Xi(z)=1$ can have solutions only if $z\in \Omega:=\{z\in\bC\;:\;|\Im(z)|\leq C_\sharp \nu_0\mu^3\leq C_\sharp\mu;\; |\Re(z)-1|\leq c C_\sharp\}$, for some constant $C_\sharp>0$. If we consider $|\tau-\tau_0|\leq c_2\mu^3$, then Lemmata \ref{lem:turn} and \ref{lem:vercond2} imply that $|\nu(\tau)-\nu_0|\leq c_1\mu$, provided $c_2$ is chosen small enough.

Next, note that, recalling  \eqref{eq:physical1},
\[
\begin{split}
\partial_z\Xi(z)&=\nu \int\alpha (\Id z-\cL_{T_{\nu(\tau),h_{\nu(\tau)}}})^{-2}\partial_\beta h_{\nu(\tau)}
=z^{-2}\nu \int\alpha \partial_\beta h_{\nu(\tau)}+\cO(\nu\mu^3)\\
&=-z^{-1}\Xi(z)+\cO(\nu\mu^3).
\end{split}
\]
Thus, for $\tau=\tau_0$ we have $\partial_z\Xi(1)=-1+\cO(\nu\mu^3)$. 
On the other hand,
\[
\begin{split}
&\partial_{\tau}\Xi(z)=-\nu'z^{-1}\int\alpha(\Id -z^{-1}\cL_{T_{\nu, h_\nu}})^{-1}\partial_\beta h_\nu\\
&-\nu z^{-2}\int\alpha(\Id -z^{-1}\cL_{T_{\nu, h_\nu}})^{-1}\left[\partial_h\cL_{T_{\nu, h_\nu}}\partial_\tau h_\nu+\partial_\nu \cL_{T_{\nu, h_\nu}}\nu'\right]\\
&\times(\Id -z^{-1}\cL_{T_{\nu, h_\nu}})^{-1}\partial_\beta h_\nu\\
&-\nu z^{-1}\int\alpha(\Id -z^{-1}\cL_{T_{\nu, h_\nu}})^{-1}\partial_\beta \partial_\tau h_\nu.
\end{split}
\]
Hence, for $\tau=\tau_0$, recalling Lemma \ref{lem:turn}, and equations \eqref{eq:basic_der}, \eqref{eq:h-der},
\[
\begin{split}
\partial_{\tau}\Xi(z)|_{\tau=\tau_0}=&-\nu_0 z^{-2}\int\alpha(\Id -z^{-1}Q_{\nu_0})^{-1}\partial_\beta \left[\cL_{T_{\nu, h_\nu}}(\Id -z^{-1}Q_{\nu_0})^{-1}\partial_\beta h_\nu\right]\\
&\times \int\alpha \Theta(\nu_0,h_{\nu_0})
+\nu_0 z^{-1}\int\alpha(\Id -z^{-1}Q_{\nu_0})^{-1}\partial_\beta \Theta(\nu_0,h_{\nu_0})\\
=&- z^{-2}\int\alpha(\Id -z^{-1}Q_{\nu_0})^{-1}\partial_\beta \left[Q_{\nu_0}(\Id -z^{-1}Q_{\nu_0})^{-1}\partial_\beta h_\nu\right]\\
&+\nu_0 z^{-1}\int\alpha(\Id -z^{-1}Q_{\nu_0})^{-1}\partial_\beta \Theta(\nu_0,h_{\nu_0}).
\end{split}
\]
Recalling \eqref{eq:Q}, \eqref{eq:Theta}, \eqref{eq:h} and \eqref{eq:ab}, we obtain
\[
\begin{split}
\partial_{\tau}\Xi(z)|_{\tau=\tau_0}=&\nu_0 z^{-1}\int\alpha \partial_\beta \Theta(\nu_0,h_{\nu_0})+\cO(\mu^3)\\
=&\nu_0 z^{-1}\int[\partial_\beta^2\alpha ]  h_{\nu_0}+\cO(\mu^3)\\
=&-\nu_0 \mu z^{-1}\int\cos\langle \beta,x\rangle a(x-\nu_0\beta\omega)+\cO(\mu^2\nu_0)\\
=&-\nu_0 \mu z^{-1}\int\cos\langle \beta,x+\nu_0\beta\omega\rangle\left[\beta_1\sin\langle \beta,x\rangle-\beta_2\cos \langle \beta,x\rangle\right]+\cO(\mu^2\nu_0)\\
&=-\frac{\nu_0}2 \mu z^{-1}\left[\beta_1\cos\nu_0\beta\omega-\beta_2\sin\nu_0\beta\omega\right]+\cO(\mu^2\nu_0)\\
&=\frac{\nu_0}2 \mu z^{-1}\sin(\nu\omega-\vartheta)+\cO(\mu^2\nu_0)=\frac{\nu_0}{z\nu}\Xi(z)+\cO(\mu^2\nu_0).
\end{split}
\]
We can then apply the implicit function Theorem \ref{thm:implicit-func} which, provided that $c$ is small enough, implies that in $\Omega$, all the solutions $z(\tau)$ of $\Xi(z)=1$ satisfy
\[
\begin{split}
z(\tau)&=1+z'(\tau)(\tau-\tau_0)+\cO(\tau-\tau_0)^2=1-\frac{\partial_\tau\Xi(1)}{\partial_z\Xi(1)}(\tau-\tau_0)+\cO(\tau-\tau_0)^2\\
&=1+\partial_\tau\Xi(1)(\tau-\tau_0)+\cO(\tau-\tau_0)^2.
\end{split}
\]
Accordingly, near the turning point, we have that the only eigenvalues outside the disk of radius $\frac 12$ are $1$ and
\[
z(\tau)=1-\left\{1+\cO(\mu^2\nu_0)\right\}(\tau-\tau_0)+\cO(\tau-\tau_0)^2.
\]
 It follows that at a turning point $z$ crosses the unit circle, and hence, only one branch can be physical. 
\end{proof}

\subsection{Phase transition}\label{sec:phase}\ \\
{\em Phase transition} in statistical mechanics usually means a situation where finite size systems have a unique equilibrium state,  while the corresponding infinite system has more than one. It is then natural to enquire if the multiple states that we have found in our example correspond to a phase transition or not. In this subsection, we show that we are in the presence of a real phase transition.\\
We start by recalling the discussion in \cite{BLS23} about the sense in which the systems we are discussing in this paper are the limit of finite size systems. 
Let $T_0= Ax\mod 1$, $A\in\operatorname{SL}(2,\bN)$, $T=\rho T_0^{n_*}\rho^{-1}$ as in the previous subsections. 
Moreover, we choose $\beta$ so that $T_0\beta=\lambda \beta$,  $\lambda>1$. To simplify the argument, we assume that $A^*=A$. Finally, we define $T_{N,\nu}:\bT^{2N}\to\bT^{2N}$ as
\[
(T_{N,\nu}(x))_i=T(x_i)+\frac{\nu}N\sum_{j=1}^N\beta \alpha(x_j).
\]
Suppose that $\mu_N=\frac 1N\sum_{i=1}^N\delta_{x_i}$ converges weakly to the measure $h(y) dy$. Then, for each $\vf\in \operatorname{Lip}( \bT^{d},\bR)$ we have
\begin{align}
\lim_{N\to\infty}\int\vf\circ  (T_{N,\nu}(x)) \mu_N&=\lim_{N\to\infty}\frac 1N\sum_{i=1}^N\vf\left(T(x_i)+\frac{\nu}N\sum_{j=1}^N \beta\alpha(x_j)\right)\nonumber\\
&=\lim_{N\to\infty}\frac 1N\sum_{i=1}^N\vf\left(T(x_i)+\nu\beta\int \alpha h\right)\label{eq:dynamicfi}\\
&=\int \vf\circ\left(T(x)+\nu\beta\int \alpha h\right) h(x) dx\nonumber\\
&=\int \vf\cL_{T_{\nu,h}} h=\int \vf \widetilde \cL_{\nu} h.\nonumber
\end{align}
Thus, $\widetilde \cL_{\nu}$ can be interpreted as the evolution of the density of particles for the infinite system arising from the dynamics $T_{N,\nu}$, see \cite{K00} for more details. \\
The question about phase transition can then be recast as: 
\begin{itemize}
\item {\em how many physical measures does $(\bT^{dN}, T_{N,\nu})$ have? }
\end{itemize}
To answer, we need to understand better the map  $T_{N,\nu}$.
\begin{lemma} There exist $\mu_\star<\frac 12$ such that, for each each $N\in\bN$, $|\mu|\leq \mu_\star$, $\nu\in\bR$, the map $T_{N,\nu}$ is Anosov.
\end{lemma}
\begin{proof}
Let $W=\{(\beta u_1, \dots, \beta u_N)\;:\; u_i\in\bR\}$. Let $\pi$ be the orthogonal projection on $W$ and 
\[
\cC_+=\{v\in\bR^{dN}\;:\; \|(\Id-\pi) v\|\leq \sqrt \mu \|\pi v\|\}.
\]
Let $v\in\cC_+$, then we can write $v=a+\sqrt \mu b$, where $a\in W$ and $b\in W^\perp$. Then, recalling \eqref{eq:rhoinv},
\[
\begin{split}
(D T_{N,\nu} v)_i&= (D\rho DT_0^{n_*} D\rho^{-1} (a+\sqrt \mu b))_i+\frac{\nu}N\sum_{j=1}^N\beta \langle \nabla\alpha(x_j), v_j\rangle\\
&= (D\rho DT_0^{n_*}[ \tilde a+\sqrt\mu \tilde b])_i+\frac{\nu}N\sum_{j=1}^N\beta \langle \nabla\alpha(x_j), v_j\rangle\\
&= (D\rho [\lambda^{n_*} \tilde a+\lambda^{-n_*}\sqrt \mu\tilde b])_i+\frac{\nu}N\sum_{j=1}^N\beta \langle \nabla\alpha(x_j), v_j\rangle
\end{split}
\]
where $\tilde a\in W$, $\tilde b\in W^\perp$ with $\|a-\tilde a\|\le C\mu$ and $\|b-\tilde b\|\le C\sqrt\mu$. Since, by the Cauchy-Schwarz inequality,
\[
\left\|\left(\frac 1N\sum_{j=1}^N\beta \langle \nabla\alpha(x_j), v_j\rangle\right)\right\|^2\le N^{-1}\left(\sum_{j=1}^N\| \nabla\alpha(x_j)\|\| v_j\|\right)^2\leq C\|v\|^2,
\]
we get
\[
\left\|\pi \left[D T_{N,\nu} v\right]\right\|\geq\frac{ \lambda^{n_*}}2\|v\|-C|\nu|\|v\|.
\]
Accordingly, for $|\nu|\leq c\mu^{-2}$ for $c$ small enough and choosing $n_*$ such that $\lambda^{-n_*}\leq\mu^3$,
\[
\begin{split}
&\left\|\pi \left[D T_{N,\nu} v\right]\right\|\geq \frac{ \lambda^{n_*}}4\|v\|\\
&\left\|(\Id-\pi) \left[D T_{N,\nu} v\right]\right\|\leq \lambda^{-n_*}\|v\|+\Const\mu\lambda^{n_*}\|v\|\leq \Const\mu\lambda^{n_*}\|v\|\\
&\phantom{\left\|(\Id-\pi) \left[D T_{N,\nu} v\right]\right\|}
\leq \const C\mu \left\|\pi \left[D T_{N,\nu} v\right]\right\|\leq \sqrt \mu\left\|\pi \left[D T_{N,\nu} v\right]\right\|
\end{split}
\]
for $\mu$ small enough. Therefore, $\cC_+$ is an invariant cone and, for each $v\in \cC_+$, $\|DT_{N,\nu}v\|\geq \frac{\mu^{-3}}8\|v\|$.
Similar computations apply to the cone
\[
\cC_-=\{v\in\bT^{dN}\;:\;\|\pi v\| \leq \sqrt \mu \|(\Id-\pi)v\| \}
\]
and the map $T_{N,\nu}^{-1}$. Thus, $T_{N,\nu}$ is an Anosov map.
\end{proof}
Since the map $T_{N,\nu}$ is topologically  mixing \cite[Proposition 18.6.5]{KH95}, it follows that the push forward of any absolutely continuous measure converges to the SRB measure $\mu^N_{SRB}$, e.g. \cite[Theorem 2.3]{GL06}. In turn, this implies that for Lebesgue almost all $x\in\bT^{2N}$
\begin{equation}\label{eq:srb}
\lim_{n\to\infty}\frac{1}n\sum_{k=0}^{n-1}\delta_{T_{N,\nu}^k(x)}=\mu^N_{SRB}.
\end{equation}
That is $\mu^N_{SRB}$ is the unique physical measure.\footnote{ Recall that in this finite dimensional setting, for a measure $\mu$ to be physical, it is enough that the following holds
\begin{equation}\label{eq:phys}
\lim_{n\to\infty}\frac{1}n\sum_{k=0}^{n-1}\delta_{T_{N,\nu}^k(x)}=\mu
\end{equation} 
for a set of points with positive Lebesgue measure.}
\begin{remark}
Note that the definitions of physical measure for the finite and infinite system differ substantially  (compare Definition \ref{def:physical} with footnote 18). However, they both have a fundamental character in common: they refer to measures that are physically observable.
\end{remark}
Let us connect the finite-dimensional case and the infinite-dimensional one. For each $\vf\in\cC^0(\bT^2,\bR)$, defining $\pi_i:\bT^{2N}\to\bT^2$ as $\pi_i(x)=x_i$, we can write
\[
\frac{1}N\sum_{i=1}^N\vf(x_i)=\left[\frac{1}N\sum_{i=1}^N\vf\circ\pi_i\right](x).
\]
This means that if $x\in\bT^{2N}$ is distributed according to an exchangeable measure $\mu^N$,\footnote{ Recall that an exchangeable measure is a measure on a product space which is invariant under finite permutations of the coordinates.} and we are interested in observables $\Phi_N$ of the form 
\[
\Phi_N(x)=\left[\frac{1}N\sum_{i=1}^N\vf\circ\pi_i\right](x), 
\]
with $\vf\in\cC^0(\bT^2,\bR)$, then
\[
\int \Phi_N(x) \mu^N(dx)=\int \vf(y) (\pi_1)_*\mu^N(dy)\\
\]
\begin{lemma}
The SRB measure $\mu^N_{SRB}$ is exchangeable.
\end{lemma}
\begin{proof}
Equation \eqref{eq:srb} implies that, for Lebesgue almost all $x\in\bT^{2N}$,
\[
\lim_{n\to\infty}\frac{1}n\sum_{k=0}^{n-1}\left[\frac{1}N\sum_{i=1}^N\vf\circ\pi_i\right](T^n_{N,\nu}(x))=\frac{1}N\sum_{i=1}^N\mu^N_{SRB}(\vf\circ\pi_i).
\]
On the other hand if $\sigma:\{1,\dots, N\}\to\{1,\dots, N\}$ is a permutation, then $\sigma\circ T_{N,\nu}=T_{N,\nu}\circ \sigma$, hence for each $\psi, h\in\cC^0(\bT^{2N},\bR)$ and $\int h=1$, 
\[
\begin{split}
\mu^N_{SRB}(\psi\circ\sigma)&=\lim_{n\to\infty}\int\psi\circ \sigma\circ T_{N,\nu}^n h=\lim_{n\to\infty}\int\psi\circ T_{N,\nu}^n h\circ\sigma^{-1}\\
&=\mu^N_{SRB}(\psi)\int h\circ\sigma^{-1}=\mu^N_{SRB}(\psi).
\end{split}
\]
Thus, $\sigma_*\mu^N_{SRB}=\mu^N_{SRB}$, that is the measure $\mu^N_{SRB}$ is exchangeable. 
\end{proof}

Accordingly, $(\pi_1)_*\mu^N_{SRB}$ is the unique measure that plays the role of the physical invariant measure for the system of size $N$ and for the relevant observables.
It follows that we are witnessing a real phase transition: the \emph{non uniqueness} of the physical invariant measure is an infinite dimensional phenomenon.
\appendix

\section{A quantitative Implicit Function Theorem}\label{sec:implicit}
Since it is not readily available in the literature, in this appendix we provide, for the reader's convenience, a quantitative version of the implicit function for functions between Banach spaces.

Let $\cB$, and $\cP$ be Banach spaces and let
\[
F\in   \cC^0(\cB\times \cP,\cB).
\] 
be Fréchet differentiable.
Let $(x_0,\lambda_0)\in \cB\times\cP$ such that $F(x_0,\lambda_0)=0$. For each $\delta>0$ let $V_{\delta}=\{(x,\lambda)\in\cB\times\cP\}\;:\; \|x-x_0\|_{\cB}\leq \delta, \|\lambda-\lambda_0\|_{\cP}\leq \delta\}$.

\begin{theorem} \label{thm:implicit-func} Assume that $F(x_0,\lambda_0)=0$, $\partial_{x}F(x_0,\lambda_0)$ is invertible and assume there exists $\Upsilon,\delta>0$ such that 
\[
\begin{split}
&\sup_{(x,\lambda)\in V_\delta}\|\Id -[\partial_{x}F(x_0,\lambda_0)]^{-1} \partial_{x}F(x,\lambda)\|_{\cB}\leq \frac 12\\
&\sup_{(x,\lambda)\in V_\delta}\|[\partial_{x}F(x_0,\lambda_0)]^{-1} \partial_{\lambda}F(x,\lambda)\|_{\cB}\leq \Upsilon.
\end{split}
\]
Let $\Delta:=\{\lambda\in\cP\::\;\|\lambda-\lambda_0\|_\cP\leq \frac{\delta}{2\Upsilon}\}$. Then there exists  $g\in\cC^0(\Delta,\cB)$, Fréchet differentiable, such that all the solutions of the equation $F(x,\lambda)=0$ in the set $\cK:=\left\{(x,\lambda)\in\cB\times\cP\;:\; \|\lambda-\lambda_0\|_\cP\leq \frac{\delta}{2\Upsilon}, \|x-x_0\|_{\cB}\leq \delta\right\}$ are given by $\{(g(\lambda),\lambda)\;:\;\lambda\in\Delta\}$. In addition,
\[
\partial_\lambda g(\lambda)=-(\partial_x F(g(\lambda),\lambda))^{-1}\partial_\lambda F(g(\lambda),\lambda).
\]
\end{theorem}
\begin{proof}
Let  $A(x,\lambda)=\partial_{x}F(x,\lambda)$, $M=\|A(x_0,\lambda_0)^{-1}\|$.

Let $\lambda$ be such that $\|\lambda-\lambda_0\|_\cP<\delta_1\leq\delta$. Consider $U_\delta=\{x\in\cB\;:\;\|x-x_0\|_{\cB}\leq \delta\}$ and the function $\Theta_\lambda:U_\delta \to \cB$ defined by
\begin{equation}\label{eq:Newton-imp}
\Theta_\lambda(x)=x-A(x_0,\lambda_0)^{-1}F(x,\lambda).
\end{equation}
If $F(x,\lambda)=0$ then $\Theta_\lambda(x)=x$, while if $\Theta_\lambda(x)=x$, $F(x,\lambda)=0$. 
Next,
\[
\|\partial_x\Theta_\lambda\|=\|\Id-A(x_0,\lambda_0)^{-1}A(x,\lambda)\|\leq \frac 12
\]
for all $(x,\lambda)\in V_\delta$. So
\[
\|\Theta_\lambda(x_0)-x_0\|=\|\Theta_\lambda(x_0)-\Theta_{\lambda_0}(x_0)\|\leq \left\|\int_0^1\partial_\lambda\Theta_\xi(x_0)(\lambda-\lambda_0) d\xi\right\|\leq \Upsilon\|\lambda-\lambda_0\|_{\cP}.
\]
Thus, for $\|x-x_0\|\leq \delta$ and $\|\lambda-\lambda_0\|_{\cP}\leq \frac{\delta}{2\Upsilon}$, we have
\[
\begin{split}
\|\Theta_\lambda(x)-x_0\|&\leq \|\Theta_\lambda(x)-\Theta_\lambda(x_0)\|+\|\Theta_\lambda(x_0)-\Theta_{\lambda_0}(x_0)\|\\
&\leq \frac 12 \|x-x_0\|+\Upsilon|\lambda-\lambda_0|\leq \delta.
\end{split}
\]
It follows that $K_\delta=\{x\in\cB\;:\;\|x-x_0\|\leq \delta\}$ is an invariant set for all $\Theta_\lambda$, such that $\|\lambda-\lambda_0\|_{\cP}\leq \frac{\delta}{2\Upsilon}$.
In addition, setting $\delta_1:=\min\left\{\delta, \frac{\delta}{2\Upsilon}\right\}$, for all $\lambda\in \{\lambda\;:\;\|\lambda-\lambda_0\|_\cP\leq\delta_1\}=:\Delta$,
\[
\|\Theta_\lambda(x)-\Theta_\lambda(y)\|\leq \frac 12\|x-y\|.
\]
The existence and uniqueness of $x\in K_\delta$ such that $\Theta_\lambda(x)=x$ follows then by the standard contraction fixed point Theorem. We have obtained a function $g:=\Delta\to\cB$ such that $F(g(\lambda),\lambda)=0$.

It remains the question of regularity. Let $\lambda,\lambda'\in \Delta$. By \eqref{eq:Newton-imp}
\[
\begin{split}
\|g(\lambda)-g(\lambda')\|&\leq \|\Theta_\lambda(g(\lambda))-\Theta_\lambda(g(\lambda'))\|+\|\Theta_\lambda(g(\lambda'))-\Theta_{\lambda'}(g(\lambda'))\|\\
&\leq \frac 12\|g(\lambda)-g(\lambda')\|+\Upsilon|\lambda-\lambda'|.
\end{split}
\]
This yields the Lipschitz continuity of the function $g$ with Lipschitz constant $2\Upsilon$. To obtain the differentiability we note that, by the differentiability of $F$ and the above Lipschitz continuity of $g$, for $\ve\in\cP$ small enough, we have
\[
\begin{split}
\|F(g(\lambda+\ve), \lambda+\ve)-F(g(\lambda),\lambda)+\partial_xF&[g(\lambda+\ve)-g(\lambda)]+\partial_\lambda F \ve\|=o(\|\ve\|_\cP)\\
&+o(\|g(\lambda+\ve)-g(\lambda)\|)=o(\|\ve\|_\cP).
\end{split}
\]
Since $F(g(\lambda+h), \lambda+h)=F(g(\lambda), \lambda)=0$, it follows
\[
\lim_{\ve\to 0}\|\ve\|_\cP^{-1}\|g(\lambda+\ve)-g(\lambda)+[\partial_xF]^{-1}\partial_\lambda F \ve\|=0
\]
which concludes the proof of the Theorem.
\end{proof}

\end{document}